\documentclass[11pt]{amsart}
\usepackage{amsmath, amsthm, mathabx,amscd, amsfonts, amssymb, hyperref, mathrsfs, textcomp, epsfig, a4wide,graphicx, IEEEtrantools, tikz-cd, verbatim, xypic,enumerate,adjustbox}
\usepackage{macros}
\usetikzlibrary{matrix}

\theoremstyle{remark}

\usepackage[colorinlistoftodos]{todonotes}

\newcommand{\SelfArrow}{\raisebox{3pt}{\adjincludegraphics[valign=m, height=6pt]{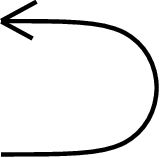}}}
\newcommand{\HFred}{HF^{\mathrm{red}}}
\renewcommand{\HF}{\mathbf{HF}}

\begin{document}

\title{On $h$-cobordisms of complexity $2$}

\author{Roberto Ladu}

\address{Fakult\"{a}t f\"{u}r Mathematik, Ruhr Universit\"{a}t Bochum,  Universit\"{a}tsstrasse 150, D-44801 Bochum, Germany} 
\email{roberto.ladu.math@gmail.com}

\begin{abstract} We study $5$-dimensional $h$-cobordisms of Morgan-Szab\'o complexity $2$. We compute the
 monopole Floer homology and the action of the twisting involution of the protocork boundary associated with such $h$-cobordisms,
obtaining an obstruction for  $h$-cobordisms between  exotic pairs to have minimal complexity.
We construct the first examples of $h$-cobordisms of non-minimal, in fact, arbitrarily large, complexity between an exotic pair of closed, $1$-connected $4$-manifolds. Further applications include strong corks. 
\end{abstract}

\maketitle
\newcommand{\ellinv}{\tau^{-}}
\newcommand{\cC}{\mathcal{C}}
\section{Statement of results.}
Morgan and Szab\'o associate to a $5$-dimensional $h$-cobordism $Z:X_0\to X_1$ with $\pi_1(Z) = 1$,  an even, non-negative number $\cC(Z)$ called complexity \cite{MorganSzabo99}. $Z$ is a cylinder if and only if $\cC(Z)=0$, otherwise $\cC(Z)\geq 2$. In this paper we study $h$-cobordisms  with $\cC(Z) =2$. In this case $X_0$ and $X_1$ are related by a twist of the protocork $P_0$ in Figure~\ref{Pic:P_0} \cite{Ladu}.
Understanding how this operation affects Seiberg-Witten invariants boils down to computing the monopole Floer homology of $Y^3:=\partial P_0$ \cite{KM} and the action of the twisting involution $\tau$ on $\HMfrom_{-1}(Y)$.
\begin{theorem}\label{thm:HMYTauAction} The Floer homology of $Y$ is supported in the torsion \spinc structure and 
\begin{equation*}
	\HMfrom_\bullet(Y) \simeq \hat \cT_{(-3)} x_3 
		\oplus \hat \cT_{(-2)} \langle x_1, x_2\rangle \oplus \hat \cT_{(-1)} x_0 \oplus \F_{(-1)}\langle \Delta,\alpha_1\rangle \oplus \F_{(-2)}\langle \beta_1,\beta_2\rangle.
\end{equation*}
Under this isomorphism, $x_0  = \HMfrom(P_0\setminus \ball)(\hat 1)$ and  $\tau_*$ acts on $\HMfrom_{-1}(Y)$ as
\begin{align*}
&\begin{bmatrix}
				1& 0 & 0\\
				1 & 1& 0\\
				0 & 0 & 1 \\
			\end{bmatrix}: \F x_0\oplus \F \Delta \oplus \F \alpha_1 \to\F x_0\oplus \F \Delta \oplus \F \alpha_1
\end{align*}
In particular $\tau_*(x_0) = x_0 + \Delta$ and 
$U\cdot \Delta = 0$.
\end{theorem}

Theorem~\ref{thm:HMYTauAction}, gives obstructions to  $\cC(Z)=2$, we apply this in item(\ref{thm:complexity1}) below to produce the first examples of $h$-cobordisms $Z$ with $\cC(Z)>2$ between a pair of closed exotic $4$-manifolds.
\begin{theorem}\label{thm:complexity}	  	Let $X_0 $ be the elliptic surface $E(1)$ and let $X_1$ be a Dolgachev surface or the result of Fintushel-Stern's knot surgery on $E(1)$.
	\begin{enumerate}
		\item \label{thm:complexity1}Then  there are at least $2^{17}17!$ $h$-cobordisms $X_0\#17 \bar \CP^2 \to X_1\# 17 \bar\CP^2$ of complexity strictly larger than $2$ (hence at least $4$).
		\item \label{thm:complexity2} Moreover, $\forall N > 0, \exists m_N \in \N$ such that there are at least $2^{m_N}m_N!$ $h$-cobordisms from $X_0\# m_N\bar\CP^2 $ to $ X_1\#m_N\bar \CP^2$ of complexity larger than $N$.
	\end{enumerate}
\end{theorem}
This improves \cite[Thm. 1.1]{MorganSzabo99} in two ways.  Firstly,  \cite[Thm. 1.1]{MorganSzabo99}  shows the existence of a sequence of inertial (i.e. $X_0=X_1$) $h$-cobordisms $Z_N:E(1)\#m_N\bar\CP^2\to E(1)\# m_N\bar\CP^2$ of diverging complexity, we extended this to non-inertial ones in item(\ref{thm:complexity2}).
Secondly, given $N>0$ the results of \cite{MorganSzabo99} alone are not enough to determine for which $m_N$  the condition $\cC(Z_N)>N$ is achieved. We show that in the case $N=2$,  $m_N = 17$ is sufficient thanks to Theorem~\ref{thm:HMYTauAction} and to the fact
that there is only one protocork to study in this case.

\paragraph{} Our second application regards corks. A cork is a pair $(C,f)$ where $C$ is a contractible  $4$-manifold and $f:\partial C\to \partial C$ is an involution that does not extend to a diffeomorphism of $C$. The Akbulut cork \cite{Akbulut91} was the first example discovered.  One motivations behind the study of corks
is that any  pair of $1$-connected $4$-manifolds $(X_0, X_1)$ which is exotic (i.e. $X_0\simeq_\mathrm{C^{0}} X_1$, $X_0\not\simeq_\mathrm{C^\infty} X_1$), is related by a cork twist operation i.e.  $X_1$ can be obtained from $X_0$ by removing  a copy of some cork   and gluing
it back via the boundary involution \cite{CHFS,Matveyev}. 
The involution $f$ is said to be \emph{strongly non-extendable} (\emph{SNE} for short) if it does not extend to any $\Z/2$-homology ball bounding $\partial C$ \cite[Thm. D]{LinRubermanSaveliev2018}. The Akbulut's cork involution is SNE  \cite{LinRubermanSaveliev2018}. A sufficient condition for being SNE, which applies to most known corks, is given in \cite{DaiHeddenMallick}. The first example of cork
with a not SNE involution was given by \cite{HaydenPiccirillo}.

Any cork is supported by a protocork \cite[Prop. 2.16]{Ladu}, thus it is natural to wonder if
strong non-extendability descends from some property of the involution of a supporting protocork.
More specifically one can ask if any cork supported by $P_0$ has a SNE involution. 
We did not manage to prove this in general, but we have the following partial result.

\begin{corollary}\label{cor:StrongCork}  Let $Y'$ be an homology sphere with an involution $\tau'$.  Suppose that  there exists an equivariant cobordism $Q:Y\to Y'$, where $Y=\partial P_0$ is equipped with the protocork involution $\tau$.
In addition suppose that the cobordism map induced by $Q$ restricts to an isomorphism, with $\F = \Z/2$-coefficients, 
$\HMfrom(Q; \F): \HMfrom_{-1}(Y)\to\HMfrom_{-1}(Y')$, in particular it is a homomorphism of degree $0$.
Then $(Y', \tau')$ does not extend to any $\Z/2$-homology ball.
\end{corollary}

Two examples of corks which are supported by $P_0$ and satisfy the hypothesis of Corollary~\ref{cor:StrongCork}, are 
 the Akbulut cork and the manifold on the right of \cite[Fig. 4]{DaiHeddenMallick} with parameters $n=1, n_1= 0, n_2 = 1$.
 The latter has boundary  $(Y_{1,\infty, 1}, \tau_{1,\infty,1})$ in Figure~\ref{Pic:Yabc}.

\subsection{Further comments and related work.} $Y$ is the  graph manifold with plumbing graph given by two spheres with trivial normal bundle, intersecting thrice with sign $+,-,+$.  $\HMfrom_\bullet(Y)$ is isomorphic to the Heegaard-Floer homology  $\HF^-(Y)$ \cite{HM=HF_1}, and for a large class of graph manifolds,  $\HF^-$ can be computed via general methods \cite{OSPlumbed,NemethiAR, NemethiLattice, OSSLattice, Zemke2021Equivalence}.  Moreover Hanselman devised an algorithm to compute $\widehat{HF}$ for every graph manifold \cite{HanselmanGraphManifolds}.
However, since our graph has cycles and we are interested in the $\HF^-$ variant, these results do not cover our case.
Graphs similar to ours appear in \cite{SunukjianThesis}, and \cite{Horvat}. The former considers plumbing graphs arising from 
two spheres with trivial normal bundle intersecting \emph{algebraically} zero times while in the latter the spheres have
normal bundles of Euler number $n,m\geq 4$ and intersect \emph{geometrically} twice with the same sign.

We carry out the computation reducing it to understading,  the Floer homology of the boundary of the first positron cork \cite{AkbulutMatveyevConvex}. The latter homology was computed with $\F$-coefficients in \cite{Guth2024Exotic}. 
The computation of the action of $\tau$ relies on  \cite[Thm. A ]{LinRubermanSaveliev2018} and equivariant surgery exact sequences. Understanding the maps in the exact sequences requires us to compute the Floer homology of 
 a certain homology sphere, this is done following an approach due to \cite{HalesThesis}.

Theorem~\ref{thm:HMYTauAction} tells us about how to glue up Seiberg-Witten invariants along the protocork $P_0$, using $\tau$ as gluing map \cite{KM}. This can be compared with \cite{MorganSzabo96}, which instead gives information about the variation in Seiberg-Witten invariants 
depending on the embedding of protocorks and not on the protocork used (the type of \cite{MorganSzabo96} is not the complexity of \cite{MorganSzabo99}).

Theorem~\ref{thm:complexity} is similar in spirit to \cite{MorganSzabo99}.  The requirement that the ends of the $h$-cobordism
are closed manifolds, is important. Indeed, with different methods,  we produced examples of $h$-cobordisms of diverging complexity between non-diffeomorphic, $1$-connected, $4$-manifolds with non-empty boundary  already in \cite[Thm. 1.5]{LaduUniversal}. 
However this result, similarly to \cite{MorganSzabo99}, does not give explicitly the exotic pair possessing an $h$-cobordism 
of complexity larger than $N$, for a given $N$.
Upper bounds on the complexity of inertial $h$-cobordisms between rational surfaces follow from the work of Schwartz  \cite{SchwartzComplexity}.

Finally, we remark that the original version of this paper used $\Z$-coefficients,  using $\F$ instead avoids having to worry about signs in the exact triangle and is compatible with the version proved in monopole Floer homology, which is in \cite{MonopolesAndLensSpaceSurgeries}.
A mild refinement of Theorem~\ref{thm:HMYTauAction} involving $\Z$-coefficients can be found in Proposition~\ref{prop:refinement} below.

\subsection{Structure of the paper.} We set up our notation in  section~\ref{Sec:Notation}. The proof of Theorem~\ref{thm:HMYTauAction} occupies section~\ref{Sec:ComputingHMfromY} to \ref{sec:Computation-11Infty}. In section~\ref{Sec:ComputingHMfromY}, we compute $\HMfrom_\bullet(Y)$.
In section~\ref{Sec:ActionInvolution}, we compute the action of the involution, this relies on Claim~\ref{claim:HFY-11infty} which is proved in the next section. 
In section~\ref{sec:Computation-11Infty} we compute $\HMred(Y_{-1,1,\infty};\Z)$ (see Figure~\ref{Pic:Yabc}), resorting to computing the Floer homology of a surgery on a $(1,1)$-knot and using the mapping cone formula \cite{OzsvathSzaboMappingCone}.
In order to spell clearly all the conventions adopted and to introduce our notation, this section begins with a brief recap of the method of \cite[Sec. 6.2]{OzsvathSzaboKnotFloer} to compute the $CFK^\infty$ complex of a $(1,1)$-knot.

In section~\ref{Sec:ApplicationComplexity} we give the proof  of Theorem~\ref{thm:complexity}. The interested reader 
can jump directly to this section as we only use Theorem~\ref{thm:HMYTauAction} at the very end of the proof.
Finally, in section~\ref{Sec:CorollaryProof} we prove Corollary~\ref{cor:StrongCork}.
In  \autoref{Appendix} we prove a lemma in commutative algebra that we invoke in multiple occasions.

\vspace{0.3cm}
\paragraph{}
\textit{Acknowledgements.} The author would like to thank Sudipta Ghosh, Marco Golla, Riccardo Piergallini and Hugo Zhou for helpful conversations and Irving Dai for telling him about \cite{Guth2024Exotic} and helpful discussions.  This project was mostly carried out during the author's stay at Max Planck Institute for Mathematics in Bonn and the author is thankful for  its hospitality and financial support.

\section{Notation conventions}\label{Sec:Notation}

When not specified differently, the 	Floer homology groups have $\F:= \Z/2$ coefficients.
We will use the following shorthand notation to denote the towers in monopole Floer homology: $\hat \cT := \F[[U]]$,
$\bar \cT := \F[[U, U^{-1}]$, $\check \cT := \F[U,U^{-1}]/U\F[U,U^{-1}]$. These $\F[U]$-modules are graded with 
$1$ in degree $0$ and $U\cdot$ being an homomorphism of degree $-2$.
 When working with Heegaard-Floer homology, we will use the notation $\cT^- =\hat \cT, \cT^+= \check \cT$ and $\cT^\infty = \bar \cT$.
Notice that  we will use the completed Heegaard-Floer homology groups $\HF^-, \HF^\infty, \HF^+$  \cite{ManolescuOzsvath} in order to use the Floer exact triangle for the variants $-,\infty$.

Let $R$ be a $\Q$-graded ring,  with $1\in R$, and $d\in \Q$. Then $R_{(d)}$ will denote the same module but  with gradings shifted so that the unit has degree $d$. 

If   $x_1, \dots, x_n $, $n\in \N$ are formal variables,  $R_{(d)}\langle x_1,\dots, x_n\rangle$ will denote $ (R_{(d)})^n$ and we will represent  the element $(r_1,\dots, r_n) \in R_{(d)}^n$ as $\sum_{i=1}^n r_i x_i$, $r_i \in R_{(d)}$. When $n=1$, we will drop the parenthesis and write, for example, $R_{(d)}x_1$.

If $W$ is a $4$-manifold with boundary, we denote by  $W\setminus \ball$ the manifold obtained by removing a ball from $W$. We can interpret it as a cobordism $\SS^3\to \partial W$, then the relative Seiberg-Witten invariant
of $W$ associated with $\sstruc \in Spin^c(W)$, is given by $\HMfrom(W\setminus \ball, \sstruc)(\hat 1) \in \HMfrom_\bullet(\partial W)$
where $\hat 1 \in \HMfrom_\bullet(\SS^3)$ is the standard generator.

For $a,b,c \in \Z\cup \{\infty\}$, define the $3$-manifold $Y_{a,b,c}$, via the surgery presentation showed in \autoref{Pic:Yabc}(right).
Notice that $Y \simeq Y_{0,\infty, 0}$.

\section{Computing $\HMfrom_\bullet(Y)$.}\label{Sec:ComputingHMfromY}
We begin with the following lemma.
\begin{lemma}\label{lemma:TorsionSpincStructures}
	$\HMfrom_\bullet(Y_{a,b,c}, \sstruc) = (0)$  for every $\sstruc \in Spin^c(Y_{a,b,c})$ with $c_1(\sstruc)$ non-torsion.
\end{lemma}
\begin{proof}
		From the surgery presentation we obtain that $H_1(Y_{a,b,c})\simeq \frac{\Z}{a\Z} \mu_a\oplus \frac{\Z}{b\Z} \mu_b \oplus \frac{\Z}{c\Z} \mu_c$ 
		where $\mu_a$, $\mu_b$, $\mu_c$ are meridians to the blue unknots in \autoref{Pic:Yabc}(right).
		The Poincar\'e dual of $c_1(\sstruc)$ is  $PD(c_1(\sstruc)) = \alpha \mu_a + \beta \mu_b + \kappa \mu_c$ for some $(\alpha, \beta, \kappa) \in \Z/a\oplus \Z/b\oplus \Z/c$. 
		
		If $c_1(\sstruc)$ is non-torsion, then at least one of $a,b,c$ must be $0$ and the corresponding coefficient $\alpha, \beta, \kappa$
		is non-zero. We deal with the case when $a=0$, the other cases are similar.
		
		From Figure~\ref{Fig:Torus_a} we see that there is a torus $\Sigma_a$, geometric dual to $\mu_a$.
		Now since 
		\begin{equation}
			 |\alpha| = |\langle c_1(\sstruc), \Sigma_a\rangle | > 2 g(\Sigma_a) -2 = 0,
		\end{equation}
		\cite[Cor. 40.1.2]{KM} implies the thesis.\end{proof}
\begin{figure}
\begin{center}
\includegraphics[scale=0.3]{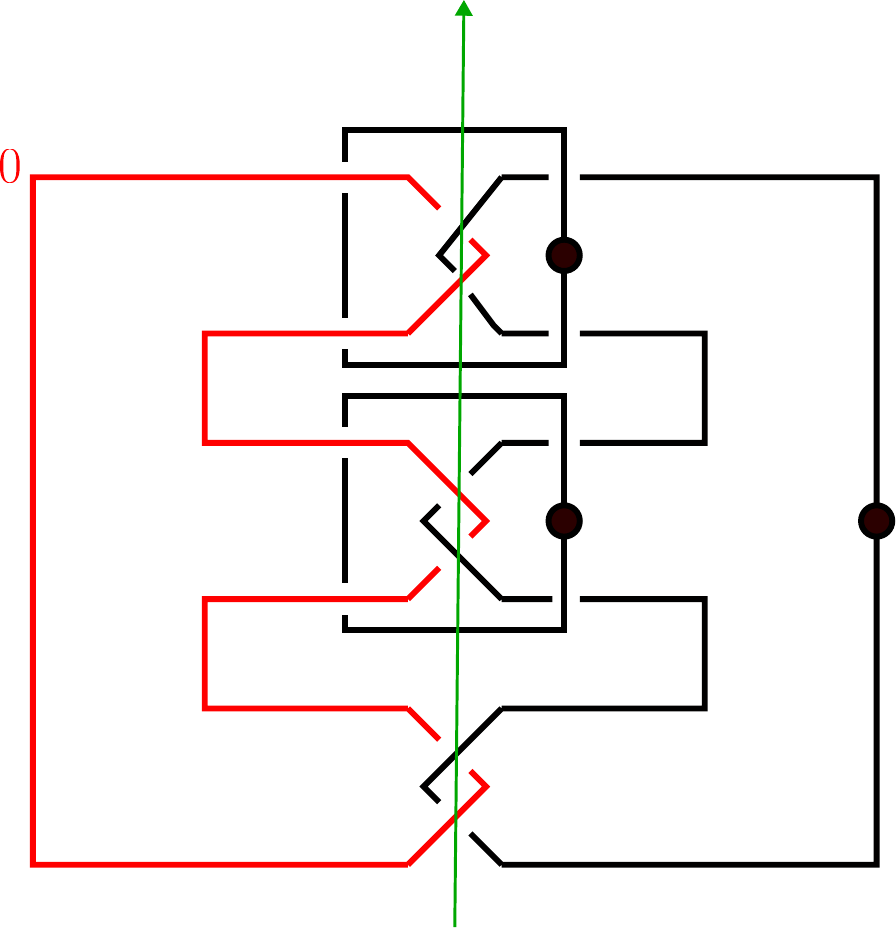} \quad \quad \includegraphics[scale=0.3]{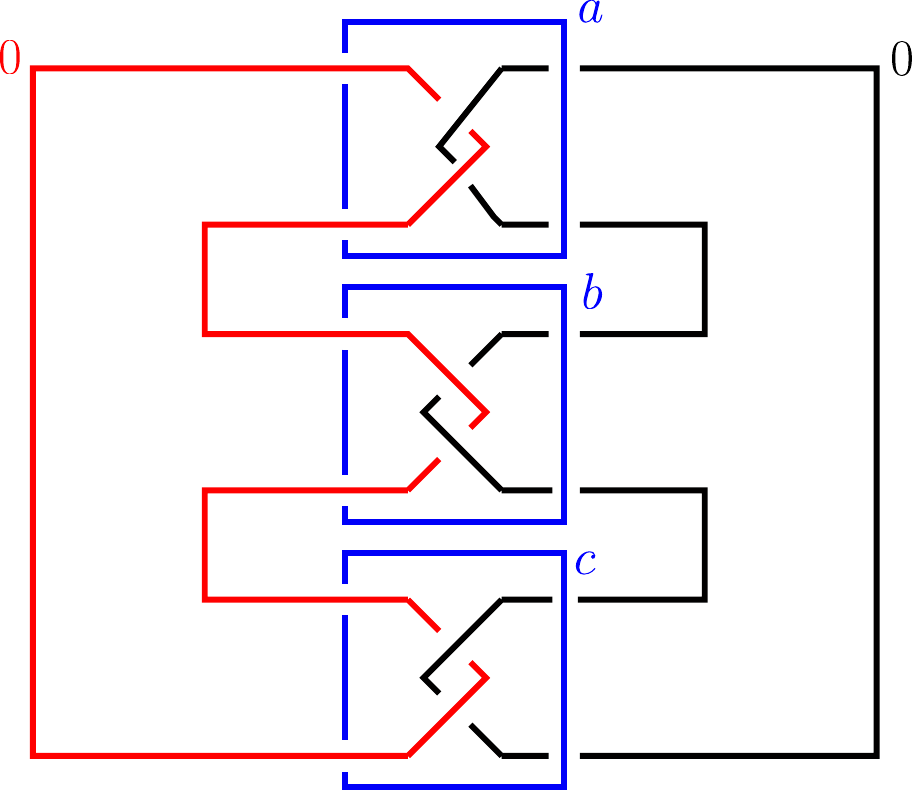}
\end{center}
\caption{\label{Pic:P_0} Left: Kirby diagram for the protocork $P_0$, $Y= \partial P_0$, the involution $\tau:Y\to Y$ is induced by a rotation by $\pi$ in the vertical axis.
\label{Pic:Yabc} Right: surgery presentation of the $3$-manifold $Y_{a,b,c}$. Notice that $Y\simeq Y_{0,0, \infty}\simeq Y_{0,\infty,0}\simeq Y_{\infty, 0,0}$.}
\end{figure}

\begin{figure}
\begin{center}
 \includegraphics[scale=0.3]{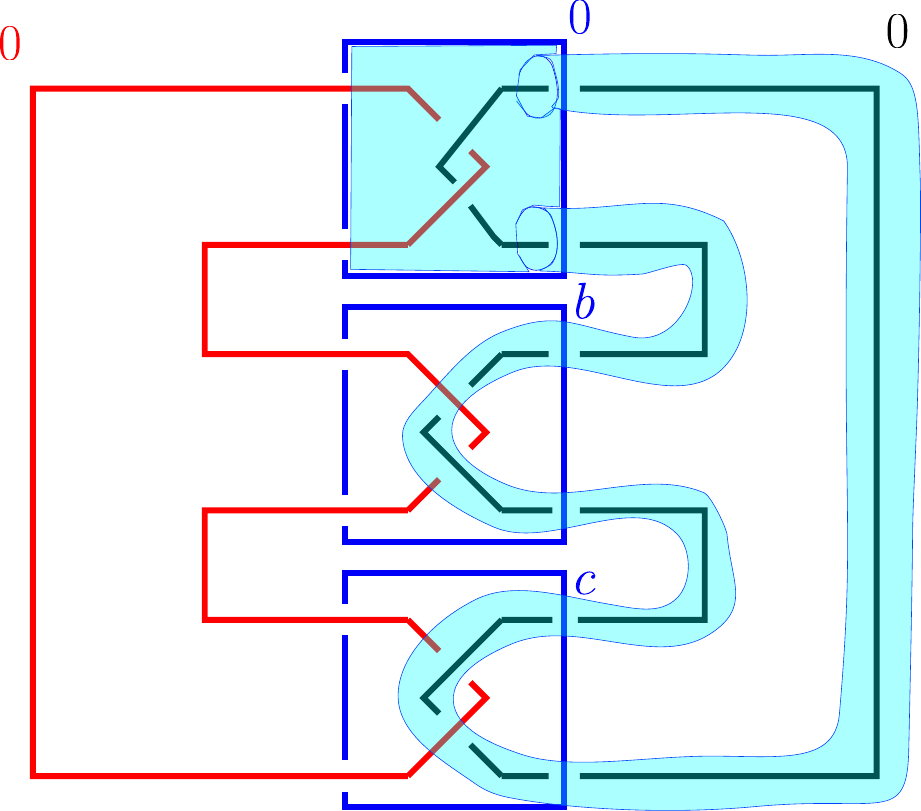}
\end{center}
\caption{\label{Fig:Torus_a} The torus $\Sigma_\alpha$ is obtained by capping the surface of genus one in blue with the disk coming 
from the Dehn filling of the $0$-framed unknot.}
\end{figure}

By \cite[Thm. 1.4]{Ladu}, $\HMfrom_{\bullet}(Y) \simeq \hat \cT_{(-3)}\oplus  (\hat \cT)_{(-2)}^2 \oplus  \hat \cT_{(-1)}\oplus \HMred(Y)$ thus the only thing we need to compute is the reduced homology.

\begin{proposition}\label{prop:determinationUpToDegShift}$\HMred(Y)\simeq \HMred(Y_{-1,\infty, -1})[-2]$.
\end{proposition}
\begin{proof}  Note that $Y \simeq Y_{0,\infty,0}$.
The sequence of surgeries shown in Figure~\ref{Pic:SurgeriesSequences} induces two knot exact triangles in $\HMfrom_\bullet$ and $\HMbar$.
We have that $Y_{\infty, \infty,0}\simeq \SS^1\times \SS^2$, $Y_{-1, \infty, \infty}\simeq \SS^3$, $Y_{-1,\infty, -1}$ is an homology sphere  and, ignoring the grading, the Floer homology of the other manifolds are given by:
\begin{align*}
	\HMfrom_\bullet(Y_{0,\infty, 0})\simeq (\hat\cT)^4\oplus \HMred(Y_{0,\infty, 0}) & & \HMfrom_\bullet (Y_{-1,\infty, 0})\simeq (\hat \cT)^2\oplus \HMred(Y_{-1,\infty, 0}),
\end{align*}
indeed they are both have vanishing triple cup product.
The vertical sequence of Figure~\ref{Pic:SurgeriesSequences}  gives rise to  a commutative diagram

\begin{tikzcd}[row sep=small, column sep = small]
(\hat\cT)^2\oplus \HMred(Y_{-1,\infty, 0})\arrow[r]\arrow[d,"\simeq"] & \hat\cT  \arrow[r]\arrow[d,"\simeq"]& \hat \cT\oplus \HMred(Y_{-1,\infty, -1})  \arrow[r]\arrow[d,"\simeq"]& \cdots\\
	\HMfrom_\bullet(Y_{-1,\infty,0})\arrow[r] \arrow[d] & \HMfrom_\bullet(\SS^3) \arrow[r] \arrow[d]& \HMfrom_\bullet(Y_{-1,\infty,-1}) \arrow[r] \arrow[d]& \cdots\\
		\HMbar(Y_{-1,\infty, 0}) \arrow[r] \arrow[d,"\simeq"] & \HMbar(\SS^3) \arrow[r] \arrow[d,"\simeq"] &  \HMbar(Y_{-1,\infty, -1}) \arrow[r] \arrow[d,"\simeq"]  &\cdots \\
(\bar\cT)^2\arrow[r] & \bar\cT   \arrow[r]& \bar\cT \arrow[r] & \cdots.\\
\end{tikzcd}

Now an application of Lemma~\ref{lemma_algebraicCaseb0} below proves that ignoring the grading
\begin{equation}\label{eq:isoWithoutGradingPropIniziale}
	\HMred(Y_{-1,\infty,-1})\simeq \HMred(Y_{-1,\infty,0}).
\end{equation}
 Lemma~\ref{lemma_algebraicCaseb1}  applied to the horizontal sequence gives $\HMred(Y_{-1,\infty,0})\simeq \HMred(Y_{0,\infty, 0})$.
Notice that in order to apply the Lemmas in the Appendix, we need to check that the grading assumptions \textbf{AG} are satisfied (see Definition~\ref{ass:AG}). This is straightforward for $Y_{-1,\infty, -1}$ since it bounds a contractible $4$-manifold and in the case of $Y_{-1,\infty,0}$
we can run an argument as in \cite[Thm. 1.4]{Ladu}.

 Now the isomorphism $\HMred(Y_{0,\infty, 0})\simeq \HMred(Y_{-1,\infty,-1})$
 is induced by two cobordisms $C_1: Y_{-1,\infty, -1}\to Y_{-1,\infty, 0}$ and $C_2:Y_{-1,\infty, 0}\to Y_{0,\infty, 0}$ 	obtained by attaching $2$-handles along null-homologous knots equipped with the Seifert-framing.
	Consequently both the cobordism maps have degree $-1$ \cite[eq. 28.3]{KM} and therefore
	\begin{equation}
		\HMred(Y)\simeq  \HMred(Y_{-1,\infty, -1})[-2],
	\end{equation}
	as graded $\F[U]$-modules.
  \end{proof}

\begin{figure}
\begin{center}
\includegraphics[scale=0.6]{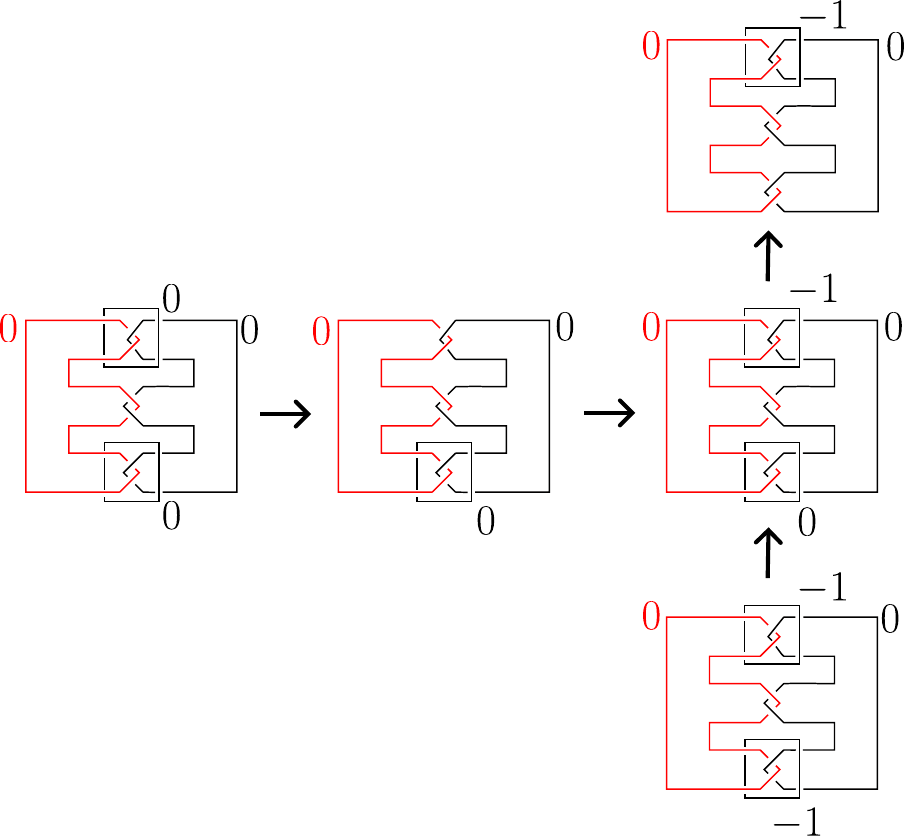}
\end{center}
\caption{\label{Pic:SurgeriesSequences} The horizontal sequence is the sequence of manifolds appearing in the  long exact triangle sequence induced by surgery along  the $a$-framed curve: $Y_{0,\infty, 0}\to Y_{\infty, \infty, 0}\simeq \SS^1\times \SS^2\to Y_{-1,\infty, 0}\to\cdots$. The vertical sequence is induced by surgery along the  $c$-framed curve: $ Y_{-1, \infty, -1}\to Y_{-1,\infty, 0} \to Y_{-1,\infty, \infty}\simeq \SS^3  \to \cdots$ }
\end{figure}

The manifold $Y_{-1,\infty, -1}$ is the mirror of the first positron cork \cite{AkbulutMatveyevConvex}.
The Heegaard-Floer homology of $Y_{-1,\infty,-1}$ is computed   with $\F$-coefficients by Guth \cite{Guth2024Exotic}, who actually computes the Floer homology of $-Y_{-1,\infty,-1}$ in \cite[Figure 4.2]{Guth2024Exotic} and also by Hales in \cite[pg. 73]{HalesThesis} where however the grading is not correct.

\begin{proposition}\label{prop:HFred-1Infty-1_F_coefficients}
		$\HFred(Y_{-1,\infty,-1};\F) \simeq \F^2_{(1)}\oplus \F^2_{(0)}$.
\end{proposition}
\begin{proof}  It is enough to compute $HF^{red}(-Y_{-1,\infty,-1};\F)$ in virtue of \cite[Prop. 28.3.4]{KM} and \cite{HM=HF_1}.
	$-(Y_{-1,\infty,-1})$ is the boundary of the first positron cork \cite{AkbulutMatveyevConvex}, this is also $+1$-surgery
	on a knot $J\subset \SS^3$ \cite[remark 7.8]{DaiMallickStoffregen}. The large surgery complex for $J$ is shown in 
	\cite[Figure 4.2]{Guth2024Exotic}, and since $J$ has genus $1$ this computes $HF^-(\SS^3_{+1}(J))$.
\end{proof}

Putting Proposition~\ref{prop:determinationUpToDegShift} and Proposition~\ref{prop:HFred-1Infty-1_F_coefficients} together  we conclude the proof that
	\begin{equation}\label{eq:HMfromY}
			\HMfrom_\bullet(Y) \simeq \hat \cT_{(-3)}\oplus  (\hat \cT)_{(-2)}^2 \oplus  \hat \cT_{(-1)}\oplus \F^2_{(-1)}\oplus \F^2_{(-2)}.
	\end{equation}

With a bit more work	we can say something about $\HMred(Y;\Z)$.
\begin{proposition}\label{prop:refinement} 
	$\HMred(Y;\Z) \simeq \Z_{(-1)}\oplus \Z/2^{k_1}\Z_{(-1)} \oplus  \Z_{(-2)}\oplus \Z/2^{k_2}\Z_{(-2)}  \oplus T^{odd}$
	where $T^{odd}$ is a finitely generated $\Z$-module consisting of odd torsion elements, 
	$k_1,k_2\in \{\N_{>0}, \infty\}$ (where $\Z/2^\infty \Z:=\Z$), and  $k_1 = \infty$  if and only if  $k_2 = \infty$.
	Furthermore, the generator of $\Z_{(-1)}$ is an integral lift of $\Delta$.
		
\end{proposition}
\begin{proof}  The fact that $\Z<\HMred_{-1}(Y;\Z)$ follows from the fact that the protocork twist operation using $P_0$ 
is able to change the integral Seiberg-Witten invariants, indeed it relates $K_3\#\bar\CP^2$ and $3\CP^2\#20\bar\CP^2$ (see the proof of Proposition~\ref{prop:ActOfTauOnHM-1}). Thus over $\Z$, $\tau_* (x_0) = x_0 + \tilde{\Delta}$,
with $\tilde{\Delta} \neq 0 \in \HMred_{-1}(Y;\Z)$ an indivisible element because \eqref{eq:variationSWinvK3CP2} holds also over $\Z$. In particular $\tilde{\Delta} = \Delta \mod 2 \in \HMred_{-1}(Y;\F)$.

$Y$ possess an orientation reversing diffeomorphism \cite[Sec. 2.5]{Ladu}. Therefore 
from Poincar\'e duality \cite[28.3.4]{KM} we infer that $\HMred_{j}(Y;\Q)\simeq\HMred_{-j-3}(Y;\Q)$ for any $j$.
Now the claim follows from the universal coefficient theorem and \eqref{eq:HMfromY}.
\end{proof}

\section{Action of the involution on $\HMfrom_\bullet(Y)$.}\label{Sec:ActionInvolution}

In \cite[Sec.2.5]{Ladu} we define an orientation preserving involution  $\tau$ of $Y$. $\tau$ can also be described
via the surgery presentation of \autoref{Pic:P_0} as the involution induced by a $\pi$-rotation around the axis shown in the picture.  In order to study the action of $\tau$ we will make use of equivariant cobordisms.  Given $3$-manifolds $M_0,M_1$ and involutions $f_0,f_1$ we call equivariant cobordism a cobordism
$W:M_0 \to M_1$ endowed with an involution $\hat f: W\to W $ extending $f_0,f_1$. Since we use $\F$-coefficients, the cobordism map $\HMfrom(W)$ will interwine $(f_0)_*$ and $(f_1)_*$ (the same applies to $\HMbar$ or $\HMto_\bullet$).

\begin{figure}
	\begin{center}
		\includegraphics[scale=0.6]{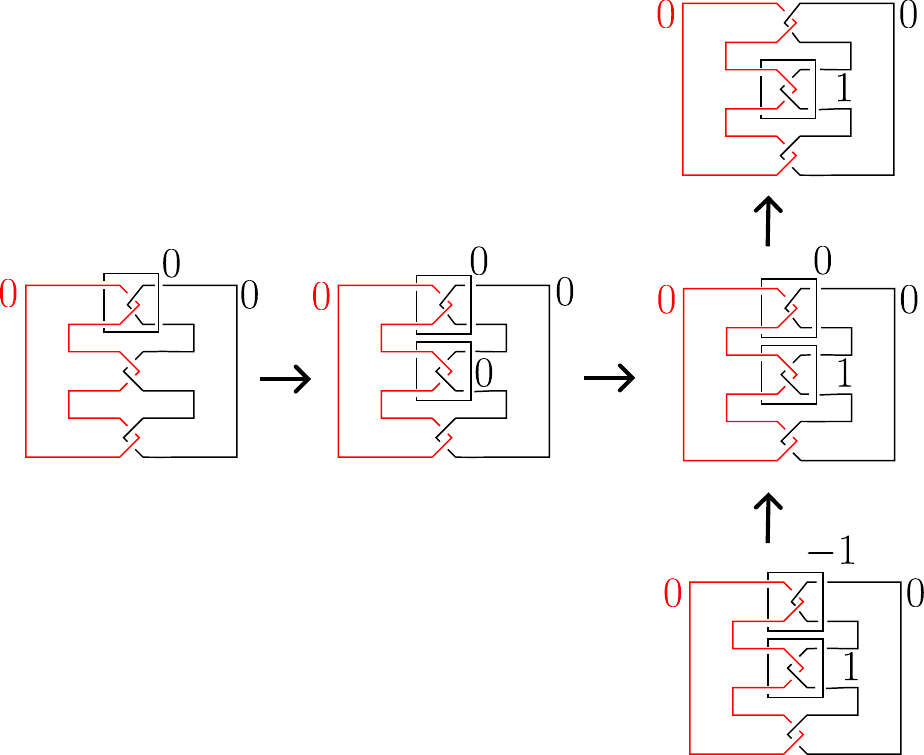}
	\end{center}
	\caption{\label{Pic:SurgerySeq2} The horizontal sequence is   induced by surgery on the meridian of the $b$-framed curve $Y_{0,\infty, \infty}\simeq \SS^1\times \SS^2\to Y_{0,0, \infty}\to Y_{0,1,\infty} \to \cdots$. The vertical sequence is induced by surgery along the  $a$-framed curve: $ Y_{1, \infty, -1}\to Y_{0,1,\infty} \to Y_{\infty,1, \infty} \to \cdots$. Note that  $Y_{\infty,1, \infty} $ is the boundary of the Akbulut cork.}
\end{figure}
	
\paragraph{The key surgery sequences.} The surgery sequences in Figure~\ref{Pic:SurgerySeq2}
induce a commutative diagram of $\F[U]$-homomorphisms:

\begin{equation}\label{eq:TwoLES}
\begin{tikzcd}[row sep=small, column sep = small]
	 &	& 	\HMfrom_{\bullet}(Y_{\infty, 1,\infty}; \F) \arrow[bend left=80,swap]{dd} \\
\HMfrom_{\bullet}(\SS^1\times \SS^2; \F) \arrow[r] & \HMfrom_{\bullet}(Y_{0,0,\infty};\F) \arrow[r] & \HMfrom_{\bullet}(Y_{0,1,\infty}; \F)\arrow[u] \arrow[bend left=30,swap]{ll} \\
	 &	& 	\HMfrom_{\bullet}(Y_{-1, 1,\infty}; \F \arrow[u] 
\end{tikzcd}
\end{equation}
where  the two triangles are exact\cite[Sec. 42.2]{KM}.

\begin{remark}\label{remarkSeq}The involution $\tau$ extends over the cobordisms inducing the maps in Figure~\ref{Pic:SurgerySeq2} because the $2$-handles are attached equivariantly. Therefore every manifold $Y_{a,b,c}$ appearing in the diagram is endowed with an involution $\tau_{a,b,c}$ and
the maps in the diagram become equivariant.
\end{remark}

The Floer homology of the manifolds appearing in the diagram are as follows: $Y_{\infty, 1, 	\infty}$ is the Akbulut cork \cite{Akbulut91}
hence $\HMfrom_{\bullet}(Y_{\infty,1,\infty}) \simeq \hat \cT_{(-1)}\oplus \F^2_{(-1)}$ \cite{AkbulutDurusoy}.
$Y_{0,0,\infty}\simeq Y$ thus its Floer homology is given by \eqref{eq:HMfromY}.
In section~\ref{sec:Computation-11Infty}, we will compute:
\begin{claim}\label{claim:HFY-11infty}
	 $HF^{red}(Y_{-1,1,\infty})\simeq \F^2_{(-1)}$.
\end{claim}
Hence $\HMfrom_\bullet(Y_{-1,1,\infty}) \simeq \hat \cT_{(-1)}\oplus \F^{2}_{(-1)}$.
Now the exact sequence obtained by considering the horizontal sequence of \eqref{eq:TwoLES} 
implies that  $\HMfrom_{\bullet}(Y_{0,1,\infty}) \simeq \hat \cT_{(-2)}\oplus \hat \cT_{(-1)}\oplus \F_{(-1)}^2\oplus \F_{(-2)}^2$ (using Lemma~\ref{lemma_algebraicCaseb1}).
Notice that the cobordism map $\HMfrom_\bullet(Y_{0,0,\infty})\to \HMfrom_\bullet(Y_{0,1,\infty})$ has degree $0$ because the $2$-handle is attached along a non null-homologous knot.

\paragraph{}

\paragraph{On the involution $\tau_{\infty, 1, \infty}$.} We already observed that  $(Y_{\infty, 1, \infty}, \tau_{\infty, 1, \infty}) $ is the boundary of the Akbulut cork with its involution.  We need the results of \cite[Sec. 8]{LinRubermanSaveliev2018} about the action of  $\tau_{\infty, 1, \infty})$ in Floer homology. 
These results however exploit that $\HMred(Y_{\infty, 1, \infty};\Z) = \Z^2_{(-1)}$ as computed in \cite{AkbulutDurusoy}.  The latter computation uses the knot surgery exact sequence in Heegaard-Floer with $\Z$-coefficients \cite[Thm. 3.1]{OzsvathSzaboMappingCone} which we prefer not  to assume. We will
instead rely on the following.
\begin{proposition}[\cite{AkbulutDurusoy}, \cite{LinRubermanSaveliev2018}] \label{prop:ActionTauAkbulutCork} The Floer homology of the Akbulut cork's boundary is isomorphic to 
	\begin{equation}
		\HMfrom_\bullet(Y_{\infty, 1, \infty};\Z)\simeq \hat\cT_{-1}x_0\oplus \Z_{(-1)} \langle \Delta, \alpha\rangle \oplus T^{odd}
	\end{equation}
	where $T^{odd}<\HMred(Y_{\infty, 1, \infty};\Z)$ consists of elements of odd torsion.
	Moreover the action of $(\tau_{\infty, 1, \infty})_*$ on $\HMfrom_{-1}(Y_{\infty, 1, \infty};\F)$ is given by $\begin{bmatrix}
				1& 0 & 0\\
				1 & 1& 0\\
				0 & 0 & 1 \\
			\end{bmatrix}: \F x_0\oplus \F \Delta \oplus \F \alpha \SelfArrow $.	
\end{proposition}
\begin{proof} \cite{AkbulutDurusoy} showed that $\HMfrom_\bullet(Y_{\infty, 1, \infty};\F )\simeq \hat\cT_{-1}\oplus \F_{(-1)}^2$. The universal coefficient theorem then implies that
$\HMfrom_\bullet(Y_{\infty, 1, \infty};\Z)\simeq \hat\cT_{-1}\oplus({\Z}/{2^{k_1}\Z})_{(-1)}\oplus ({\Z}/{2^{k_2}\Z})_{(-1)}\oplus  T^{odd}$
where $k_1,k_2\in \{\N_{>0}, \infty\}$ and we define ${\Z}/{2^{\infty}\Z} = \Z$.
Now \cite[Thm. A]{LinRubermanSaveliev2018} implies that the Lefschetz number of $(\tau_{\infty, 1, \infty})_*$
acting on $\HMred(Y_{\infty, 1, \infty};\Q)$ is equal to $2$, thus the only possibility is that  $k_1 = k_2 = \infty$.

Let $f:\Z^2\to \Z^2$ be  the homomorphism
obtained by restricting $(\tau_{\infty, 1, \infty})_{-1}$ to the $\Z^2$ summand in $\HMred_{-1}(Y_{\infty, 1, \infty};\Z) $ and then composing with the projection map, since  $\tau$ preserves $\HMred(Y;\Z)$, and  $\Hom_{\Z}(T^{odd}, \Z)= \{0\}$, $f$ is an involution.

An involution acting on $\Z^2$ is, up to conjugation, either diagonalizable or  the map swapping the coordinates. 
Since we know that the trace of $f\otimes 1:\Z^2\otimes \Q\to \Z^2\otimes \Q$ is minus the Lefschetz number hence equal to $-2$, $f$ is diagonalizable.

Now the claimed form for the action of $(\tau_{\infty, 1, \infty})_{-1}$ with $\F$ coefficients follows, as in \cite{LinRubermanSaveliev2018}, from the fact that the Akbulut cork twist is able to modify the Seiberg-Witten invariants modulo $2$.
\end{proof}

\begin{proposition} \label{prop:ActOfTauOnHM-1} $\tau$ acts on $\widehat{HM}_{-1}(Y) \simeq \F   x_0 \oplus \F \Delta \oplus \F e_1$ as
			$\begin{bmatrix}
				1& 0 & 0\\
				1 &1& 0\\
				0 & 0 &1 \\
			\end{bmatrix}$ where $x_0  = \HMfrom(P_0\setminus \ball)(\hat 1)$ is the relative Seiberg-Witten invariant of the protocork $P_0$ and $\Delta, e_1$ is an $\F$-basis of $ \HMred_{-1}(Y)$.
\end{proposition}
\begin{proof} 
	We will consider  the  equivariant sequences  \eqref{eq:TwoLES} with 
	In the  vertical equivariant exact sequence
		\begin{equation}
			\widehat{HM}_0(Y_{-1,1,\infty})\to \widehat{HM}_{-1}(Y_{0,1,\infty}) \to  \widehat{HM}_{-1}(Y_{\infty,1,\infty}) \to \cdots
		\end{equation}
		$\HMred_{-1}(Y_{0,1,\infty})\simeq \F^2$ is mapped injectively  to $\HMred_{-1}(Y_{\infty,1,\infty})\simeq \F^2$
		because $\widehat{HM}_0(Y_{-1,1,\infty}) = (0) $. 
		Since $(\tau_{0,1,\infty})_*$ acts as $id$ on $\HMred_{-1}(Y_{0,1,\infty})$ by Proposition~\ref{prop:ActionTauAkbulutCork},
		and the cobordism maps are  equivariant, $(\tau_{-1,1,\infty})_* $  acts as $id$ on $\HMred_{-1}(Y_{-1,1,\infty})$.
		The  same argument applied to the sequence 
		\begin{equation}
			(0) =\widehat{HM}_0(\SS^1\times \SS^2)\to \widehat{HM}_{-1}(Y_{0,0,\infty}) \to  \widehat{HM}_{-1}(Y_{0,1,\infty}) \to \cdots
		\end{equation}
		implies that $\tau_* = (\tau_{0,0,\infty})_*$ acts as $id$ on $\HMred_{-1}(Y)\simeq \F^2$. 
		
		Let $x_0\in \widehat{HM}_{-1}(Y)$	be the relative Seiberg-Witten invariant of the protocork bounded by $Y$.
		Then $\widehat{HM}_{-1}(Y)\simeq \F x_0 \oplus \HMred_{-1}(Y) $ and 
		\begin{equation}\label{eq:TauAndDelta}
			\tau_*(x_0) = x_0 + \Delta
		\end{equation} where $\Delta \in \HMred_{-1}(Y)$ \cite[Thm. 1.1]{Ladu}. Since the Akbulut cork \cite{Akbulut91} is supported by the protocork $(P_0, \tau)$\cite[Fig. 5]{Ladu},
		by twisting an embedded copy of $P_0$ in  $K3\#\overline{\CP}^2$ we obtain $3\CP^2\# 20\overline{\CP^2}$.
		Hence denoting by $M\subset K3\#\overline{\CP}^2$ the complement of the embedding of $P_0$ with a ball removed:
		\begin{align}\label{eq:variationSWinvK3CP2}
				& 	\check  1= \HMarrow(M,\sstruc_M)(x_0) 	& 	 		 0 = \HMarrow(M,\sstruc_M)(x_0+\Delta) 
		\end{align}
		 for some $\sstruc_M \in Spin^c(M)$, so  $\HMarrow(M,\sstruc_M)(\Delta) = \check 1  \mod 2$.
		Therefore $\Delta\neq 0$ and can be extended to a basis of $\HMred_{-1}(Y)$.
\end{proof}

\section{Computation of $\HF^+(Y_{-1, 1, \infty})$.}\label{sec:Computation-11Infty}
\label{Sec:BigComputation}
\newcommand{\OS}{Ozsv\'{a}th and Sz\'{a}b\'{o} }
The strategy that we follow is roughly the same that has been used in \cite{HalesThesis} to carry
out the computation of $HF^+$ for the boundary of the corks $W_n$. Firstly we reduce the computation
to that of a surgery on a $(1,1)$-knot, we compute the $CKF^\infty$ complex using the method of \cite[Section 6.2]{OzsvathSzaboKnotFloer} and we conclude using the mapping cone formula \cite{OzsvathSzaboMappingCone}.

\subsection{Review of \OS's method.}\label{ReviewOzsvathSzaboMethod}
In \cite[Section 6.2]{OzsvathSzaboKnotFloer}, \OS describe a method to compute the  $CFK^\infty$ complex of a $(1,1)$-knot in $\SS^3$.
We briefly review it and spell out some orientation conventions implicit in \cite{OzsvathSzaboKnotFloer}.

With reference to \autoref{Figure:BraidGroupGenerators}~(b), we denote by $\overline\tau_1,\overline\tau_2, \overline\tau_3$ the generators of the braid group $B_4$, our convention for composition is shown in \autoref{Figure:BraidGroupGenerators}~(c).
Consider a knot $K(\epsilon, \sigma) \subset \SS^3$ as in \autoref{Figure:BraidGroupGenerators}~(a) where $\epsilon \in \{ \pm1\}$ and $\sigma$ induces a permutation preserving the color of the endpoints. In order to  simplify the exposition, and because these will be the only cases of interest to us, we restrain ourselves to $\sigma $ in the subgroup $\langle \overline\tau_2^2, \overline\tau_3\rangle$. 
A doubly pointed Heegaard diagram of genus one for $K(\epsilon, \sigma)$ is obtained as follows.
Begin with $(\Sigma, \alpha, \beta, z,w)$ as in \autoref{Figure:InitialDiagram}~(a), we stress that $\Sigma, \alpha $ and $\beta$
must be oriented as in the picture.  The  braid $\sigma$ acts on $\Sigma$, this action is shown in  \autoref{Figure:InitialDiagram} (b) and (c)  for the case of  $(\overline\tau_2)^{\pm 2}$ and $(\overline\tau_3)^{\pm 1} $ and then extended to $\langle \overline\tau_2^2, \overline\tau_3\rangle$ by composition.
We can therefore produce a new diagram $(\Sigma, \sigma(\alpha), \beta, z,w)$.
Now,  if  the algebraic intersection of the two curves with respect to our chosen orientation of $\Sigma$ satisfies $\sigma(\alpha)\cdot \beta = \epsilon$, then $(\Sigma,  \sigma(\alpha), \beta, z,w)$ is the sought diagram for $K(\epsilon,\sigma)$. 
Otherwise we concatenate  $ \sigma(\alpha)$ to $k\in \Z$ copies of  vertical meridian $\mu$ as in \autoref{Figure:addTwists} so that  $\alpha' :=  \sigma(\alpha)\#k \mu $ satisfies $\alpha'\cdot \beta = \epsilon$; the final diagram will then  be $(\Sigma, \alpha', \beta, z,w)$.

To check that our conventions are correct one can apply them to  the left handed trefoil $K(1, (\overline\tau_2)^{-2}\overline \tau_3 (\overline \tau_2)^2)$ (as showed in \autoref{Figure:LiftedExample}), indeed a wrong choice of conventions would lead one to compute the diagram of $K(\pm 1, (\overline\tau_2)^{2}\overline \tau_3^{-1} (\overline \tau_2)^{-2})$, i.e. the right handed trefoil ($-1$) or the unknot ($+1$) which can be easily distinguished from their complex.
We also warn the reader that the example \cite[Fig.12]{OzsvathSzaboKnotFloer} corresponds to $K(1,(\overline{\tau_2})^{-4}(\overline \tau_3)^2(\overline \tau_2)^{2})$, indeed the crossings of \cite[Fig.8]{OzsvathSzaboKnotFloer}  should be inverted.

\begin{figure}
\includegraphics[scale=0.5]{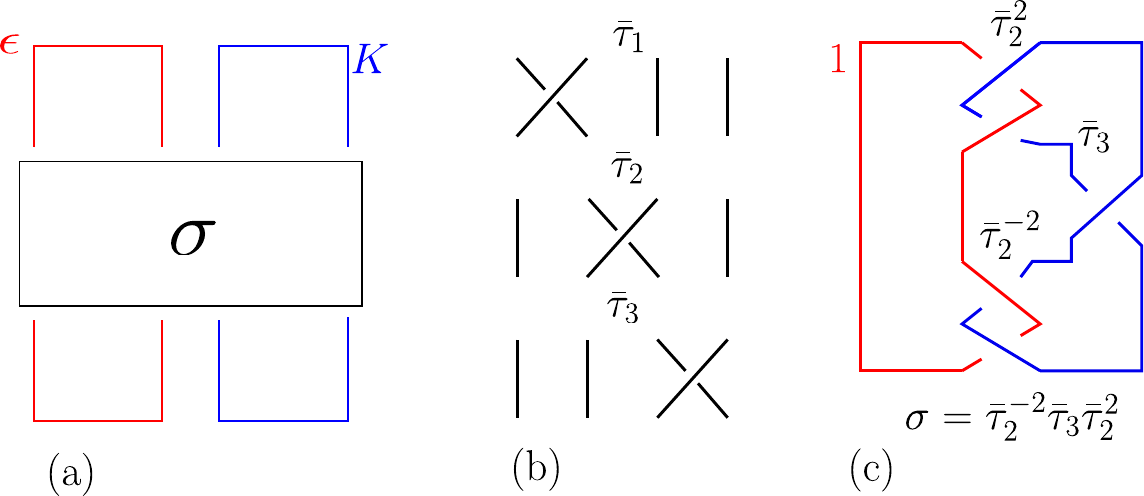}
	\caption{ 
	(a) The knot $K(\epsilon, \sigma)$ is described by the green curve, here $\epsilon \in \{\pm 1\}$ and $\sigma \in B_4$ is a product of $(\overline\tau_2)^{\pm 2}$ and $(\overline\tau_3)^{\pm 1 }$. (b) Generators of $B_4$. (c) Example to show our composition convention, here $ \sigma = (\overline \tau_2)^{-2}\overline \tau_3(\overline \tau_2)^{2}$ and $K(1,\sigma)$ is the left handed trefoil. \label{Figure:BraidGroupGenerators}} 
\end{figure}

\begin{figure}
\includegraphics[scale=0.7]{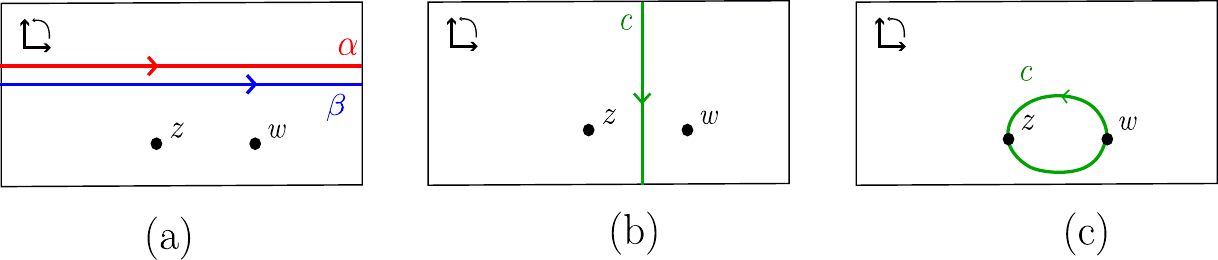}
\caption{ (a) Initial doubly pointed Heegaard diagram, the surface $\Sigma$ is the torus obtained by identifying the opposite sides of the rectangle and has  the specified orientation. (b) The action of  $(\overline \tau_2)^2$ on $\Sigma$ is a \emph{full} Dehn twist around the curve $c$. (c) The action of  $\overline \tau_3$ on $\Sigma$ is \emph{half} Dehn twist around the curve $c$. \label{Figure:InitialDiagram}}
\end{figure}

\begin{figure}
\includegraphics[scale=0.7]{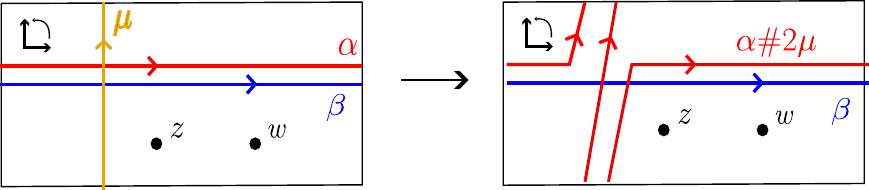}
\caption{Concatenating a multiple of the meridian $\mu$.  \label{Figure:addTwists}}
\end{figure}
\begin{figure}
\includegraphics[scale=0.7]{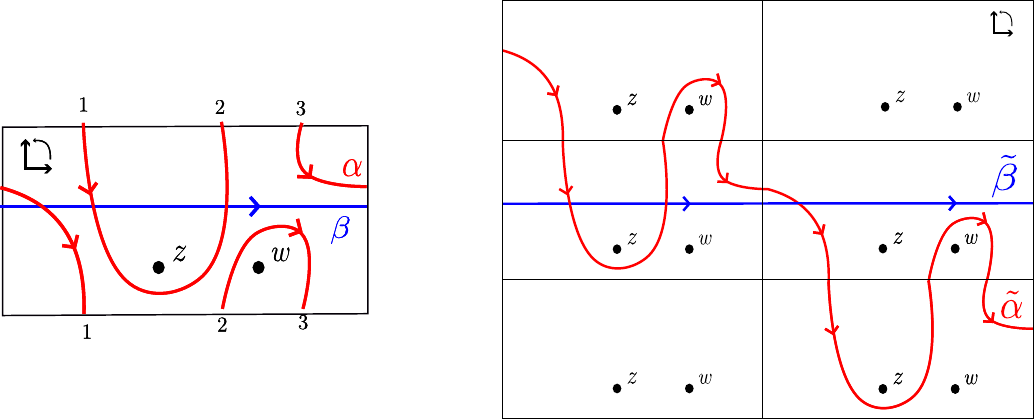}
\caption{ Double pointed Heegaard diagram (left) and lifted Heegaard diagram (right)  for example \autoref{Figure:BraidGroupGenerators}~(c).\label{Figure:LiftedExample}}
\end{figure}

Once a genus one Heegaard diagram $(\Sigma, \alpha, \beta, z, w)$ for $K(\epsilon, \sigma)$ is found, we produce a \emph{lifted} Heegaard diagram.  Denote by $\pi:\C\to \Sigma$  the universal cover  and orient $\C$ so that $\pi$ is orientation preserving.
Choose a point in $\pi^{-1}(\alpha\cap \beta)$ and use it to produce lifts $\tilde\alpha, \tilde\beta\subset \C$.
The  tuple $(\C, \tilde\alpha, \tilde \beta, \pi^{-1}(z), \pi^{-1}(w))$ is a lifted Heegaard diagram for $K(\epsilon, \sigma)$ \cite{OzsvathSzaboKnotFloer}. An example is shown in \autoref{Figure:LiftedExample}.
The complex $CFK^\infty(\SS^3, K(\epsilon, \sigma))$  is generated by $\tilde \alpha \cap \tilde \beta$ and the differential counts holomorphic Whitney disks in $(\C, \tilde \alpha, \tilde \beta)$ \cite[Sec. 6.2]{OzsvathSzaboKnotFloer}, the role of the additive assignements $n_z, n_w$ is replaced by $n_{\pi^{-1}(z)}$ and $n_{\pi^{-1}(w)}$ in the lifted setting. Notice that in the lifted diagram there are infinitely many marked points. One of the advantage of this method
is that we can determine the differential, even with $\Z$-coefficients, in a combinatorial way \cite[6.4]{OzsvathSzaboKnotFloer} \cite{GodaMatsudaMorifuji}\cite{Rasmussen2003}.

\subsection{Reduction to surgery on a knot.}
We begin by reducing the computation to a surgery on a knot in $\SS^3$.
For the scope of this section we set $\sigma :=  \bar \tau_2^2 \bar \tau_3^2 \bar \tau_2^2\bar \tau_3^{-2} \bar \tau_2^{-2}\in B_4$ and $K:= K(1,\sigma)$.

\begin{proposition}\label{prop:ReductionSurgeryOnKnot} As  graded $\F[U]$-modules 
	$\HFred(Y_{-1,1,\infty}) \simeq \HFred(\SS_{-1}^3(K))$.
\end{proposition}
\begin{proof}
	The surgeries shown in \autoref{Fig:Triangle-11Infty} give rise to an  exact triangle 
	\cite[Thm.6.2]{OzsvathSzaboRationalSurgeries}
	 \begin{equation*}
		\HF^-(Y_{-1,1,\infty})  \to \HF^- (\SS_{-1}^3(K)) \to \HF^-(\SS^1\times \SS^2)\to \HF^-(Y_{-1,1, \infty}) \to \dots
	\end{equation*}
	and an application of Lemma~\ref{lemma_algebraicCaseb0} gives us an isomorphism between the reduced homology of $Y_{-1,1, \infty}$ and $\SS_{-1}^3(K)$. It remains
	to show that this isomorphism preserves the grading.
	
	\begin{figure}
	\includegraphics[scale=0.7]{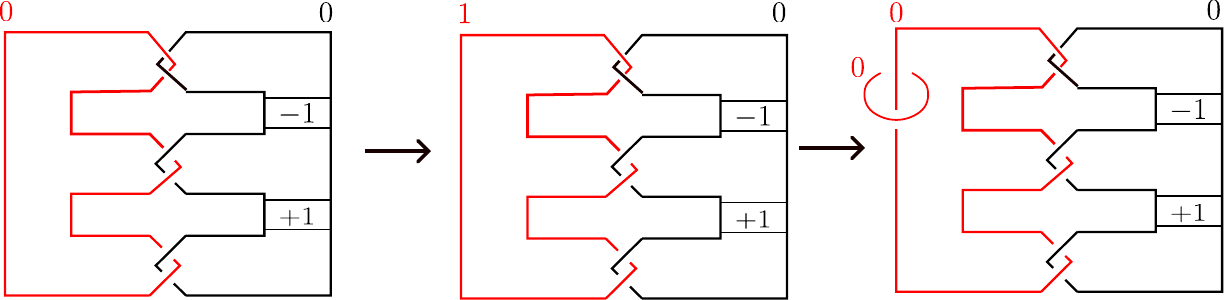}
\caption{Floer exact triangle $Y_{-1,1,\infty}\to \SS^3_{-1}(K) \to \SS^1\times \SS^2$. This surgery presentation for $Y_{-1,1,\infty}$
is obtained  from that of Figure~\ref{Pic:Yabc}~(right) by isotopying the top clasp to the bottom of the picture and then blowing down the $+1$ and $-1$-framed curves.
            \label{Fig:Triangle-11Infty}.}	
	\end{figure}

		The cobordism map $\HF^- (Y_{-1,1,\infty})\to (\SS_{-1}^3(K))$ is induced by the cobordism $W$
obtained by gluing a $2$-handle to $ Y_{-1,1,\infty}$ along the framed curve $\gamma$ as shown in \autoref{Fig:FramedGamma}~(left).
	In \autoref{Fig:FramedGamma}~(right) we show that $\gamma$ bounds a punctured torus and that the framing is equal to $-1$ with respect 	to the Seifert framing in $Y_{-1,1, \infty}$. 
	In particular the signature of  $W$ is $-1$.

		If $\sstruc$ is a \spinc structure on $W$, then $c_1(\sstruc)  = k  [S]^*$ where  $S$ is  the punctured torus capped
	with the $2$-handle of $W$ and $k\in \Z$.
	In order for $\HF^-(W,\sstruc)$ to be a non-trivial map, $\sstruc$ must satisfy the adjunction inequality \cite{OzsvathSzaboHolTriangles}:
	\begin{equation}
		2g(S)-2\geq S\cdot S + | \langle c_1(\sstruc), S\rangle |  = -1+|k|,
	\end{equation}
	therefore the only possible contribution to the cobordism map come from the \spinc structures with $k=0,\pm 1$.
	These maps have $\Q$-grading equal to $\frac {1+c_1^2(\sstruc)}{4}=\frac{1-k^2}{4	}$, so the only possibility is that $k=\pm 1$ 
	because otherwise $\HF^-(Y_{-1,1,\infty})$ would be non-trivial in some non-integral degree ($Y_{-1,1, \infty}$  bounds a contractible $4$-manifold). 
\end{proof}

	\begin{figure}
		\includegraphics[scale=0.5]{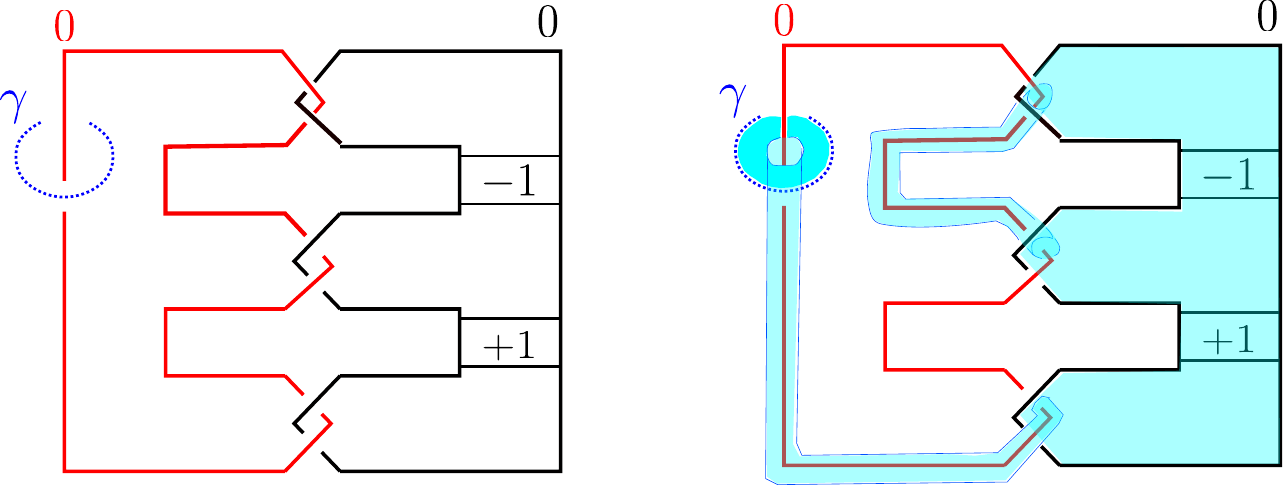}
		\caption{ Left: the framed curve $\gamma$. Right: a  Seifert surface for $\gamma$ in $\SS_{-1}^3(K)$ is obtained by capping the blu surface with the disk bounding the black curve. \label{Fig:FramedGamma}}
	\end{figure}

\subsection{Computation of $CFK^\infty(\SS^3,K)$.}
We apply \OS method to $K$. 
The resulting Heegaard diagram is drawn on the bottom-right of Figure~\ref{Fig:HD-11Infty} together with some intermediate steps to derive it so that the reader can check its correctness.
The lifted Heegaard diagram is shown in Figure~\ref{Fig:LiftedHD-11Infty}. We can find pseudo-holomorphic Whitney disks and obstruct their existence by inspecting positive domains \cite[Prop. 6.4]{OzsvathSzaboKnotFloer}. We obtain a complex where the differentials have a sign ambiguity, however in our case the complex is so simple that it will be isomorphic to that shown in Figure~\ref{Fig:Complex-11Infty} in any case.
Indeed by changing the sign of the generators we can assume that the two squares are as in Figure~\ref{Fig:Complex-11Infty} and the different possibilities for $\partial x_5$ yield isomorphic complexes.
\begin{figure}
	\includegraphics[scale=0.7]{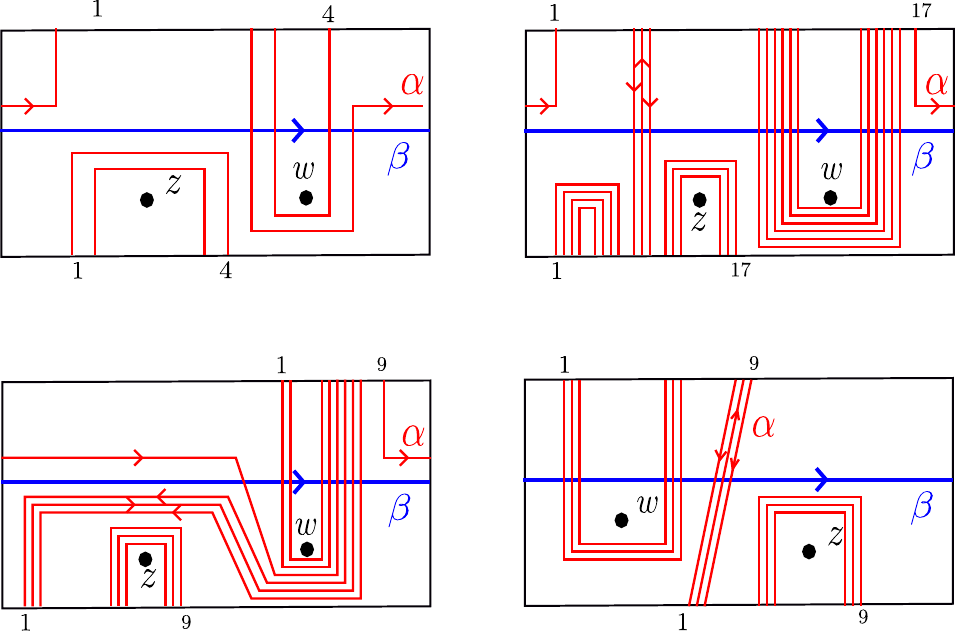}
\caption{Derivation of the Heegaard diagram for the knot $K(1,\sigma)$ on the torus. Top-left: application of $\bar \tau_2^2\bar \tau_3^{-2} \bar \tau_2^{-2}$. Top-right: application of $\bar \tau_2^2 \bar \tau_3^2 $.
Bottom-left: we perform an isotopy of the $\alpha$-curve. Bottom-right: another isotopy to get the final diagram.
            \label{Fig:HD-11Infty}.}	
\end{figure}	
\begin{figure}
	\includegraphics[scale=0.45]{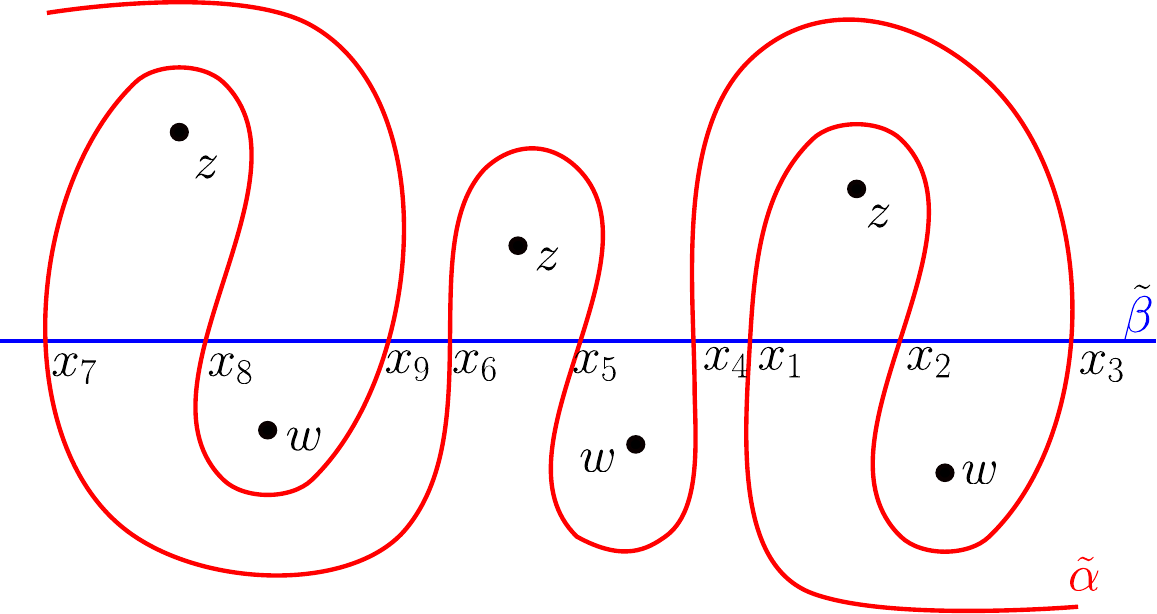}
\caption{Lifted Heegaard diagram for the knot $K(1,\sigma)$. The plane is oriented counterclockwise.
            \label{Fig:LiftedHD-11Infty}}	
\end{figure}	
\begin{figure}
	\includegraphics[scale=0.8]{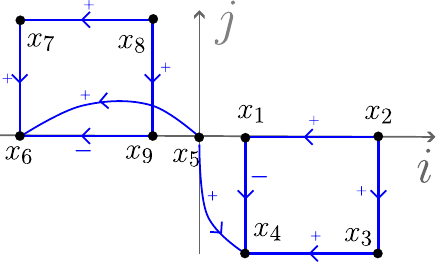}
\caption{The filtered complex $CFK^\infty(K(1,\sigma);\Z) $ is generated over $\Z[U]$ 	by the complex shown in the picture. Each arc
represents a $\pm 1$ component of the  differential, for example $\partial x_1 = -x_4$ and $\partial x_2 = x_1 + x_3$.
            \label{Fig:Complex-11Infty}.}	
\end{figure}

\subsection{Applying the mapping cone formula.}
We briefly recall how to apply \cite{OzsvathSzaboMappingCone}. 
Set  $(C, \partial) := CFK^\infty(\SS^3, K)$  and let $\cF_i, \cF_j$ be  the projections of the filtration $\cF:C\to \Z\oplus \Z$, 
so that $\cF(x) = (\cF_i(x), \cF_j(x))$ for any $x\in C$.

For any $s\in \Z$, we have quotient complexes
\begin{align}
& B^+ := C/\{x\  | \  \cF_i(x)<0\}  &  A^+_s :=  C/\{x\ | \ \cF_i(x)<0\wedge \cF_j(x)<s\} 
\end{align}
and we can  form
\begin{align}
&    \mathbb A^+ := \bigoplus_{s \in \Z}  A^+_s & \mathbb B^+ := \bigoplus_{s\in \Z} B^+\times \{s\}.
\end{align}

Since $B^+$ is a quotient complex of $A^+_s$ for any $s$, we have a projection map $v^+_s: A_s^+\to B^+$.
In \cite{OzsvathSzaboKnotFloer} it is shown the existence of a chain homotopy equivalence $C/\{\cF_j(x)<0\}\overset{\simeq}{\to} C/\{\cF_i(x)<0\}$. We can use it to construct a chain map $h_s^+: A_s^+\to B^+$ defined as the composition of the projection $A_s^+\to C/\{\cF_j(x)<s\} $
with multiplication by  $U^{-s}: C/\{\cF_j(x)<s\}\to C/\{\cF_j(x)<0\}$ composed with the above-mentioned chain homotopy equivalence.

Now for $n\in \Z\setminus \{0\}$, we define $ \mathbb D^+_{n}: \mathbb{A}^+\to \mathbb{B}^+$
\begin{align*}
  &  \mathbb D^+_n \left((a_s)_{s\in \Z}\right) = (b_s)_{s\in \Z}\\
 & b_s = h^+_{s-n}(a_{s-n})  + v_s^+(a_s).
\end{align*}

\cite[Thm. 1.1]{OzsvathSzaboMappingCone} asserts that for any $n\neq 0$ the homology of the mapping cone of $\mathbb{D}^+_n$ 
is isomorphic to $HF^+(\SS^3_n(K))$. The case of interest to us is when $n=-1$, in order to compute $H(Cone(\mathbb{D}^+_{-1}))$, we first compute  $H(B^+)$ and $H(A^+_s)$.

A direct computation using the complex in Figure~\ref{Fig:Complex-11Infty} shows that 
the homologies  are given by:
\begin{align*}
&H(B^+) = \mathcal T^+[x_1+x_5+x_9] & & H(A^+_0) =\mathcal T^+ [x_1+x_5+x_9]& \\
&H(A^+_1) = \mathcal T^+ [x_1+x_5+x_9] \oplus \F [ x_{9}]  & &H(A^+_s) = \mathcal T^+ [x_1+x_5+x_9] \text{ for $s\geq 2$ }& \\
&H(A^+_{-1}) = \mathcal T^+ U[x_1+x_5+x_9] \oplus \F U[x_1] & &   H(A^+_s) = \mathcal T^+ U^{s}[x_1+x_5+x_9] \text{ for $s\leq -2$ }.& \\
\end{align*}

We will compute $H(Cone(\mathbb D^+_{-1}))$ using the exact sequence 
\begin{equation*}
0\to \mathrm{coker} (\mathbb D^+_{-1})_*\to H(Cone(\mathbb D^+_{-1})) \to \ker(\mathbb D^+_{-1})_*\to 0.
\end{equation*} Where $(D^+_{-1})*:H(\mathbb A^+)\to H(\mathbb{B}^+) $ is the map induced
in homology by $D^+_{-1}$. Computing $(\mathbb D^+_{-1})*$ boils down to undestand the map $(h^+)_*:H(A_s^+)\to H(B^+)$, which in principle requires to understand the chain homotopy equivalence $\Phi_{ji}: C/\{\cF_j(x)<0\}\overset{\simeq}{\to} C/\{\cF_i(x)<0\}$.  In our case, it is enough to observe the following things.
Firstly, the generators $[x_9]\in H(A_{1}^+)$ and $U[x_1]\in H(A_{-1}^+)$  are killed by $h^+$ before applying $\Phi_{ji}$.
Secondly, the homology of the complex shown in Figure~\ref{Fig:Complex-11Infty} is generated by $[x_1+x_5+x_9]$
hence the homotopy equivalence  sends it to $[x_1+x_5+x_9]$.
This allows us to compute the maps $h_s^+, v_s^+$ in homology obtaining: 

\begin{tikzcd}
H(A^+_k)\arrow[d, "\simeq"]  \arrow[r, "v_k^+"] & H(B^+) \\
	\mathcal{T}^+ \oplus \F^{n(k)} \arrow[r, "1\oplus 0"]  & 	\mathcal{T}^+ \arrow[u, "\simeq"]
\end{tikzcd}
\begin{tikzcd}
H(A^+_k)\arrow[d, "\simeq"]  \arrow[r, "h_k^+"] & H(B^+) \\
	\mathcal{T}^+ \oplus \F^{n(k)} \arrow[r, " U^k"]  & 	\mathcal{T}^+ \arrow[u, "\simeq"]
\end{tikzcd}
for $k\geq 0$, 

\begin{tikzcd}
H(A^+_k)\arrow[d, "\simeq"]  \arrow[r, "v_k^+"] & H(B^+) \\
	\mathcal{T}^+ \oplus \F^{n(k)} \arrow[r, "U^{|s|}"]  & 	\mathcal{T}^+ \arrow[u, "\simeq"]
\end{tikzcd}
\begin{tikzcd}
H(A^+_k)\arrow[d, "\simeq"]  \arrow[r, "h_k^+"] & H(B^+) \\
	\mathcal{T}^+ \oplus \F^{n(k)} \arrow[r, " 1\oplus 0"]  & 	\mathcal{T}^+ \arrow[u, "\simeq"]
\end{tikzcd}
for $k< 0$, 
where $n(k)$ is the rank of the $U$-torsion part of $H(A^+_k)$. The map $(\mathbb D^+_{-1})_*$ is given by the following diagram.

\begin{tikzcd}
 	\cdots &\cT^+ \arrow[d,"U^2"]\arrow[dl, "1"] & \cT^+\oplus \F  \arrow[d,"U"] \arrow[dl, " 1\oplus 0"] & \cT^+ \arrow[d,"1"] \arrow[dl, "1"] & \cT^+ \oplus \F \arrow[d,"1\oplus 0"] \arrow[dl, " U"] & \cT^+ \arrow[d,"1 "] \arrow[dl, " U^2"]& \cdots \arrow[dl," U^3 "] \\
 	\cdots &\cT^+ & \cT^+ & \cT^+ & \cT^+ & \cT^+& \cdots 	\\
 	 \cdots	 & -2      &    -1         &0         & 1& 2 & \cdots\\
 \end{tikzcd}

Thus $\ker (\mathbb D^+_{-1})_* \simeq \F^2$ generated by $[x_9]\in H(A_1^+)$ and $U[x_1]\in H(A_{-1}^+)$ and 
$\coker (\mathbb \mathbb D^+_{-1})_* \simeq \cT^+$ generated by $y\oplus y \in H(B_{0}^+)\oplus H(B_1^+)$ where $y := [x_1+x_5+x_9]$.
 Consequently $HF^+(Y_{-1, 1, \infty})\simeq H(Cone(\mathbb D))\simeq \cT^+_{(0)}\oplus \F^2_{(-1)}$. To check that the
grading is correct notice that since $Y_{-1, 1, \infty}$ bounds a contractible $4$-manifold, the cornerstone
of the tower lies in degree $0$. Moreover $(\mathbb D^+_{-1})_*$ has degree $-1$ and $x_1,x_5,x_9$ have same bidegree in $CFK^\infty(K(1,\sigma))$ therefore the generators of the reduced part, $[x_9]$ and $[x_1]$, have grading equal to the cornerstone of the tower minus one.
This concludes the proof of Claim~\ref{claim:HFY-11infty}.

\section{Applications to complexity of $h$-cobordisms.}\label{Sec:ApplicationComplexity}
\subsection{Complexity, isometries and protocorks.}  Let $X_0$ and $X_1$ be simply-connected, closed, $4$-manifolds with intersection form  $Q_{X_0}$ and $ Q_{X_1}$. Isomorphism classes of $h$-cobordisms $X_0\to X_1$ are in bijection with isometries $Q_{X_0}\to Q_{X_1}$ \cite{Kreck2001}, thus in order to construct an $h$-cobordism it is sufficient to specify an isometry of the intersection lattice.

To an $h$-cobordism $Z:X_0\to X_1$, Morgan and Szab\'o associate a non-negative even number $\mathcal C(Z)$ called \emph{complexity} of $Z$, for a precise definition we refer to \cite{MorganSzabo99}.  Roughly, an $h$-cobordism as above, admits  \emph{normal} handle decompositions $\mathcal {D}$ with only handles of index $2$ and $3$ (see \cite[Sec. 4.1]{LaduUniversal}), the complexity of $\mathcal{D}$ is the number of excess intersections between the attaching spheres of the $3$-handles and the belt spheres of the $2$-handles in the middle level of the $h$-cobordism. $\mathcal C(Z)$ is then defined as the minimum
of the complexities among all normal handle decompositions of $Z$.

If $\mathcal{C}(Z) = 0$, then the handles cancel out and $Z$ is diffeomorphic to a cylinder. 
Morgan-Szab\'o's complexity should not be confused with the \emph{stabilization}-complexity which instead is given by the
minimal  number of $2$-handles, among normal handle decompositions.

A normal handle decompositions $\mathcal D$ for $Z$ has an associated bipartite plumbing graph $\Gamma$,  satisfying certain properties \cite[Def. 2.1]{Ladu}, representing the intersections of attaching/belt spheres of the $3$/$2$-handles.
$\mathcal D$ relates the two ends $X_0$ and $X_1$ by a cut and paste operation called \emph{protocork twist} \cite{Ladu}. The latter accounts to embed a $4$-manifold with boundary, the protocork $P_0(\Gamma)$, in $X_0$,  then removing it  and gluing in its place the so called reflection $P_1(\Gamma)$.  If $\Gamma$ has a certain symmetry then $P_1(\Gamma)\simeq P_0(\Gamma)$ and this operation is equivalent to glue back in  $P_0(\Gamma)$ via an involution $\tau \in \mathrm{Diff}(\partial P_0(\Gamma))$. 
This  happens also in the case of the protocork $P_0$ in Figure~\ref{Pic:P_0},  in this case there is only a pair of attaching and belt spheres
and they intersect thrice with signs $+,-,+$.

The isomorphism type of the protocork depends only on the isomorphism type of $\Gamma$.  Hence there are only a \emph{finite} number of protocorks with associated complexity lower than a given number. 
Here we are implicitly discarding graphs where intersections allowing to geometrically cancel a pair of handles occour.
It is key to our argument that there is only one such graph with complexity $2$. In particular if $\mathcal{C}(Z) =2$ then $X_0$ and $X_1$ are related by a protocork twist using $P_0$,  the main subject of this paper.

\subsection{Proof of Theorem~\ref{thm:complexity}.}
Let $X$ the elliptic surface $E(1)\simeq \CP^2\#9\bar\CP^2$. The regular fiber $F\subset X_0$, is contained in the Gompf nucleus $N(1)$ \cite[pg .71]{GompfStipsicz}, that we will denote by $N$, which splits $X = M \bigcup N$ where  $M$ is the complement of $N$.

Recall that $\pi_1(N) = 1$, $\pi_1(N\setminus F) = 1$, $\partial N$ is an homology sphere and $H_2(N)\simeq \Z\langle [F], [S]\rangle  $ where  $S\simeq \CP^1$  is the section of the elliptic fibration $E(1)$. With respect to these generators, $N$ has intersection form  $Q_N \simeq \begin{bmatrix}
																										0 & 1\\
																										1 & -1\\
																									\end{bmatrix}$. This gives a splitting $Q_X \simeq -E_8 \oplus Q_N$ indeed $Q_M \simeq -E_8$.
					
Let $\tilde N$ be the manifold obtained by performing either a logarithmic transform or Fintushel-Stern's knot surgery \cite{FintushelSternKnotSurgery}	on the regular fiber $F\subset N$. Then $N$ and $\tilde N$ are homeomorphic rel $\partial N$ \cite{boyer_1986} and in particular  $Q_{\tilde N}\simeq Q_N$.

 By construction, $\torus^3 \simeq \partial \nu F  $, splits $ N = N_1 \bigcup_{\torus^3} \nu F$ and $\tilde N = N_1\bigcup_{\torus^3} \tilde N_2$, where $\tilde N_2$ is homotopy equivalent to $\nu F$  relatively to the boundary \cite{KimRuberman2008} and $\pi_1(N_1) =1$.
Let $\partial: H_2(\tilde N)\to H_1(\torus^3)$ be  the connecting morphism in the Mayer-Vietoris sequence.

\begin{lemma}\label{lemma:IsometryRestrictNonZero}
There is a basis $H_2(\tilde N) \simeq \Z\langle [\tilde F], [\tilde S]\rangle$ such that 
																							$Q_{\tilde{N}} \simeq \begin{bmatrix}
																										0 & 1\\
																										1 & -1\\
																									\end{bmatrix}$,  and in addition,
																									$[\tilde F] \in Im(H_2(\torus^3)\to H_2(\tilde N))$,
																									 $\partial [\tilde S]$ generates $\ker (H_1(\torus^3)\to H_1(\tilde N_2))$
																								and $\partial[\tilde F] = 0$
\end{lemma}
\begin{proof}
	
		Let $[\tilde F]\in H_2(\tilde N)$ be the image of the generator of $H_2(\tilde N_2)\simeq H_2(\nu F)\simeq \Z$, notice that $H_2(\tilde N_2)\to H_2(\tilde N)$ is injective as $\pi_1(N_1) =1 $, hence $[\tilde F]\neq 0$.
		Moreover $H_2(\torus^3)\to H_2(\tilde N_2)$ is surjective thus $[\tilde F] \in Im(H_2(\torus^3)\to H_2(\tilde N))$.
		
		 By construction of $\tilde N_2$,  $\ker (H_1(\torus^3)\to H_1(\tilde N_2)) \simeq \Z$, let $\sigma \in H_2(\tilde N)$ be such that
		$\partial \sigma  $ is a generator of such cyclic group. The existence of $\sigma$ is guaranteed by Mayer-Vietoris.
		Since $\partial [\tilde F] = 0$		and $H_2(\tilde N)\simeq \Z^2$, $[\tilde F],\sigma$ is a basis for $H_2(\tilde N)$.
		The intersection  matrix  in this basis takes the form $\begin{bmatrix}
																										0 & \pm 1\\
																										\pm 1 & 2k-1\\
																									\end{bmatrix}$ for some $k\in \Z$. Indeed $[\tilde F]^2 = 0$, and the matrix is symmetric  and congruent to $Q_N$ hence the determinant is $-1$ and $\sigma^2$ is odd. 
		Setting $[\tilde S] :=  (\sigma\cdot\tilde F)\sigma - k [\tilde F]$ we conclude.
\end{proof}

Let $\hat\Phi:H^2(N)\to H^2(\tilde N)$ be the isometry given by Lemma~\ref{lemma:IsometryRestrictNonZero} and extend it to an isometry $\Phi: H^2(X)\to H^2(\tilde X)$ by taking direct product
with an isometry of $-E_8$. By a slight abuse of language we will denote by $\Phi$ also the associated bijection between \spinc structures.

We endow $X$ with the homology orientation given  by $H\in H^2(X)$, the class of an hyperplane and we give $\tilde X$ the homology orientation induced by $\Phi(H)$.
Let $\sstruc \in Spin^c(X)$, such that $\langle c_1(\sstruc),[F]\rangle\neq 0$ and the formal dimension is positive: $d(\sstruc)>0$.
We adopt the notation of \cite{KM}, in particular denoting as $\fm_\pm(u^{d(\sstruc)/2}| \ X,\sstruc)$ the chamber dependent Seiberg-Witten invariant \cite[Sec. 27.5]{KM}.

\begin{lemma} \label{lem:SameInvariants} Under the above hypothesis,	$\fm_{\pm}(u^{d(\sstruc)/2} | \ X,\sstruc) = \fm_{\pm}(u^{d(\sstruc)/2}| \ \tilde X, \Phi(\sstruc))$. More generally if $\Phi: H^2(N)\oplus H^2(M\# m \bar\CP^2)\to  H^2(\tilde N)\oplus H^2(M\# m \bar\CP^2)$
	is an isometry given by a direct sum of $\hat \Phi$ with an isometry of $-E_8 \oplus (-I)^{m}$, then 
	$\fm_{\pm}(u^{d(\sstruc)/2} | \ X\#m \bar\CP^2,\sstruc) = \fm_{\pm}(u^{d(\sstruc)/2}| \ \tilde X\#m \bar\CP^2, \Phi(\sstruc))$
	provided $\langle c_1(\sstruc),[F]\rangle\neq 0$ and $d(\sstruc)>0$.
\end{lemma}
\begin{proof}
	Let $k = PD[F] \in H^2(N)$ and set $M_1 := M \cup N_1$ so $X = M_1\cup N_2$ where $N_2 := \nu F$ and $\tilde X = M_1\cup \tilde N_2$. Two preliminary observations: firstly by  Mayer-Vietoris, we get that $Im(H^1(\torus^3)\to H^2(X)) = \Z k$
	hence the set of \spinc structures on $X$ restricting to two given \spinc structures on $M_1$ and $\nu F$ is an affine space over $\Z k$.
	Secondly $[S\cap \partial \nu F] \neq 0 \in H_1(\torus^3)$, we set $\eta:=[S\cap \partial \nu F]$.
	
	We will use the invariants $\fm_k(X,\sstruc)$ as defined in \cite[pg. 577]{KM}, for such invariants, the gluing formula \cite[Prop.27.5.1]{KM} gives:
	\begin{equation}\label{eq:SumMinv}
	\sum_{n\in \Z} \fm_k(u^{d(\sstruc)/2} | \ X,\sstruc + nk) e^{\langle c_1(\sstruc) + nk, [S]\rangle} = \langle U^{d(\sstruc)/2} \psi_{M_1, \sstruc|_{M_1}, S\cap {M_1}}, \psi_{N_2, \sstruc|_{N_2}, S\cap {N_2}}\rangle_{\omega} = 0,
	\end{equation}
	where on the RHS,  we have the pairing of the two relative invariants of $M$ and $N_2 = \nu F$, in $\HMfrom_\bullet(\torus^3; \Gamma_\eta)$.
	This is equal to zero because  the latter group is isomorphic to $\R$ \cite[Prop. 3.10.2]{KM}, hence annihilated by $U$.

	Since $\langle c_1(\sstruc),[F]\rangle\neq 0$, the series on the LHS is finite and, since \eqref{eq:SumMinv} holds for any positive multiple of $S$,  we can 	isolate the addenda obtaining $ \fm_k(u^{d(\sstruc)/2} | X,\sstruc + nk) = 0$ for all $n$.

	Considering $\tilde X$, $\Phi(\sstruc)$, $\Phi(k)$ and $\Phi([S])$, it still holds that $[\Phi([S])\cap \torus^3] \neq 0$ and $\Phi(PD(k))\neq 0   \in Im(H_2(\torus^3)\to H_2(\tilde X)) $ thanks to Lemma~\ref{lemma:IsometryRestrictNonZero}, hence we can apply again the gluing formula obtaining
	
	\begin{equation}
	\sum_{n\in \Z} \fm_{\Phi(k)}(u^{d(\sstruc)/2} | \ \tilde X,\Phi(\sstruc + nk)) e^{\langle \Phi(c_1(\sstruc) + nk), \Phi([S])\rangle}  = 0
	\end{equation}
	and thus $ \fm_{\Phi(k)}(u^{d(\sstruc)/2} | \tilde X,\Phi(\sstruc + nk)) = 0$  for all $n$.
	
		Now, the invariant $\fm_k(u^{d(\sstruc)/2} | X,\sstruc)$ together with $k$ and the intersection pairing of $X$, determines the invariant $\fm_{\pm}(u^{d(\sstruc)/2} | \ X,\sstruc)$. Since all these are preserved under $\Phi$, the claim follows.
		The same argument, mutatitis mutandis, generalizes to multiple blow-ups.		
		\end{proof}

Now all is set so that we can mimic the construction in \cite{MorganSzabo99}.
Let $m\geq 1$, and define
\begin{align*}
X_{0,m}:= X\# ((2m+1)^2-8)\bar{\CP}^2 & & X_{1,m}:= \tilde X\# ((2m+1)^2-8)\bar{\CP}^2.
\end{align*} 
Define  $\alpha_m = (2m+1)H - \sum_{i} E_i\in H^2(X_{0,m})$ where the sum runs over all the exceptional divisors $E_i$.
Notice that $\alpha_m^2 = -1$. Let $\rho_m:H^2(X_{0,m}) \to H^2(X_{0,m})$ be the reflection in $\alpha_m$
i.e. $\rho_m (x) = x  +2\langle x, \alpha_m\rangle\alpha_m.$

Let $\Phi:H^2(X_{0,m})\to H^2(X_{1,m})$ be an extension of $\hat \Phi$ as in Lemma~\ref{lem:SameInvariants}.
We consider the isometry $ R_m = \Phi\circ \rho_m: H^2(X_{0,m})\to H^2(X_{1,m})$, 
let   $C_m:X_{0,m}\to X_{1,m}$ be the associated $h$-cobordism \cite{Kreck2001}.

\begin{theorem}\label{thm:hCobProved}
	The complexity of the $h$-cobordism $C_m$ diverges as $m\to +\infty$.
	The complexity of $C_2$ is strictly larger than $2$. 
\end{theorem}
\begin{proof}
	First of all we study the variation of  Seiberg-Witten invariants along the $h$-cobordism $C_m$.
	Let  $\sstruc_m \in Spin^c(X_{0,m})$ be  the \spinc structure with $c_1(\sstruc_m) = \alpha_m$, and let 
$\bar \sstruc_m $ be its conjugate \spinc structure, i.e. $c_1(\bar\sstruc_m) = -\alpha_m$.
Notice that $c_1(\sstruc_m)([F]) \neq 0$ because $[F] =PD(3H - \sum_{i=1}^9 E_i)$,
and $(c_1(\sstruc_m)\cup H)[X_{0,m}] > 0$,
consequently  by Lemma~\ref{lem:SameInvariants}
	\begin{align*}
		1& = \fm_{+}(u^{d(\sstruc_m)/2}| \ X_{0,m}, \sstruc_m) = \fm_{+}(u^{d(\sstruc)/2}|\ X_{1,m}, \Phi(\sstruc_m)), \\ 
		0& = \fm_{+}(u^{d(\sstruc_m)/2}|\ X_{0,m}, \bar\sstruc_m) = \fm_{+}(u^{d(\sstruc)/2}|\ X_{1,m}, \Phi(\bar\sstruc_m)),
	\end{align*}	
where the computation of the invariants on the left follows from   $X_{0,m}$ being a rational surface.

The $h$-cobordism $C_m$ relates the \spinc structures $\sstruc $ and $R_m(\sstruc_m) = \Phi(\bar \sstruc_m)$. Since
\begin{equation}
	\fm_{+}(u^{d(\sstruc_m)/2}| \ X_{0,m}, \sstruc_m)\neq \fm_{+}(u^{d(\sstruc)/2}|\ X_{1,m}, \Phi(\bar\sstruc_m))
\end{equation} 
we observe that along the $h$-cobordism $C_m$, there is a variation of $\mod 2$ Seiberg-Witten invariants with moduli spaces of formal dimension 
$d(\sstruc_m) = m^2+m-2.$

	Now suppose by contraddiction that the complexity of $\{C_m\}_{m\in \N}$ is bounded above by $r\in \N$.
	Let $\mathcal{S}_r$ be the set of isomorphisms classes of protocorks of complexity smaller than $r$.
	$\mathcal{S}_r$ is seen to be a finite set  by enumerating  the possible plumbing graphs.
	Set \begin{equation}
			\ell := \max_{P \in \mathcal{S}_r} \{\text{$U$-torsion order of $\Delta_P$} \in \HMred_{-1}(\partial P)\} \in \N,
			\end{equation}
	where $\Delta_P$ is the difference element of the protocork $P$ \cite[Thm. 1.1]{Ladu}.
	 Then by  \cite[Corollary 1.2]{Ladu},  there cannot be any variation of Seiberg-Witten invariants corresponding to moduli spaces
	 of formal dimension larger than $2\ell$ along the $h$-cobordism $C_m$. 
	 However $d(\sstruc_m) = m^2+m-2 \to +\infty$ as $m\to +\infty$, this proves  the first part of the thesis. 
	 
	 To prove the second part, suppose  by contraddiction that $C_2$ has complexity lower or equal to $2$.
	 Then the two ends are related by a protocork twist using the protocork $P_0$ of Theorem~\ref{thm:HMYTauAction}.
	 By Theorem~\ref{thm:HMYTauAction}, and \cite[Corollary 1.2]{Ladu}, there cannot be variation in the $\mod 2$ Seiberg-Witten invariants associated
	 with moduli spaces of formal dimension $2$. On the other hand $d(\sstruc_2) = 4$. \end{proof}					
	 
\begin{remark} In \cite[Corollary 1.2]{Ladu} we assume $b^+(X) >1$, however the same proof goes through to cover the case $b^+(X)=1$ using the metric dependent $\HMarrow$-map \cite[pg. 562-566]{KM}.  Indeed the protocork
splits $X$ in $M$ and $P_0$ with $b^+(M)= b^+(X) = 1$ and $b^+(P_0)=0$  \cite[Prop. 2.4 (d)]{Ladu}
and the factorization formula \cite[(3.14)]{KM} holds using the metric dependent  $\HMarrow$-map.
\end{remark}						

\begin{remark} Dolgachev surfaces are obtained by performing two logarithmic transforms on $X$, not just one as in Theorem~\ref{thm:hCobProved}.  If we perform a second  logarithmic transform $\tilde X \leadsto\tilde {\tilde {X}} $ we can construct, in the same manner as we did before, an isometry $\tilde\Phi: H^2(\tilde X)\to H^2(\tilde {\tilde {X}})$ for which Lemma~\ref{lem:SameInvariants} holds. Indeed the only difference is that instead of using the Gompf nucleus $N$, we appeal to $\tilde N$, which has the same properties for the sake of our proof.
At this point it is sufficient to consider $h$-cobordisms induced by the composite $\tilde \Phi\circ R_m$. 
\end{remark}															
 
This proves Theorem~\ref{thm:complexity} except for the claim about the number of such $h$-cobordisms.
This can be done by 	noticing that in the construction of  $\Phi$ we can choose any isometry of $G:=-E_8\oplus (-I)^{(2m+1)^2-8}$.
We just need to find isometries of $G$ leading to distinct maps $\Phi_m=\Phi\circ\rho_m$. 
Noticing that for any $i\geq 10$, $\rho_m(2\alpha_m + E_i) =  E_i $,  and that the isometries of $(-I)^{(2m+1)^2-8}$ are signed permutations, 
we immediately get $2^{(2m+1)^2-8}((2m+1)^2-8)!$ examples. More examples can be constructed exploting the isometries of $H^2(N)$ and $-E_8$ but we decided not to pursue this in this paper.

\section{Proof of Corollary~\ref{cor:StrongCork}}\label{Sec:CorollaryProof}
\begin{proof} With the notation of Corollary~\ref{cor:StrongCork}, since $\HMfrom(Q;\F)$ is an equivariant isomorphism,
\begin{equation}
	\HMfrom_{-1}(Y'; \F)\simeq \hat \cT_{(-1)} x_0'\oplus \F_{(-1)}\Delta' \oplus \F_{(-1)}\alpha',
\end{equation} 
 and the action of $\tau'_*$ can be represented by the matrix on the left in Theorem~\ref{thm:HMYTauAction}.
Now the proof follows the same lines of \cite[Thm. D]{LinRubermanSaveliev2018}. 
Suppose that $C$ is a $\Z/2$-homology sphere bounding $Y'$ such that $\tau'$ extends to $C$ (we will denote the extension by the same name).
Let $\sstruc$ be the unique $\mathrm{spin}$ structure of $C$. 
We have  $ ax_0' + b\Delta' + c\alpha' =\HMfrom(C\setminus \ball, \sstruc; \F)(\hat 1) \in \HMfrom_{(-1)}(Y; \F')$, for some $a,b,c \in \F$. Since $C$ is negative definite, $a \neq 0$.
Since $\tau'$ extends  to $C$ and preserves the $\mathrm{spin}$ structure, $(\tau_*' -id_*) (ax_0' + b\Delta' + c\alpha')= 0 \in \HMfrom_{(-1)}(Y';\F) $. On the other hand, 
$(\tau'_*- id_*)(ax_0' + b\Delta' + c\alpha') = a \Delta' \neq 0$. 
\end{proof}

We claimed that  the Akbulut cork and $(Y_{1,\infty, 1}, \tau_{1,\infty,1})$ satisfy the hypothesis of Corollary~\ref{cor:StrongCork}.
In the case of the Akbulut cork $W_1$, we can attach equivariantly two $2$-handles to $Y$, obtaining a sequence of cobordisms $Y=Y_{0, 0,\infty}\to Y_{0,1,\infty}\to Y_{\infty, 1 \infty} = \partial W_1$ as in Figure~\ref{Pic:SurgerySeq2}.

For the case of $Y_{1, \infty, 1}$, we replace the vertical sequence of Figure~\ref{Pic:SurgerySeq2} with Figure~\ref{Fig:SeqForY1Infty1}.

\begin{figure}
\begin{center}
\includegraphics[scale=0.8]{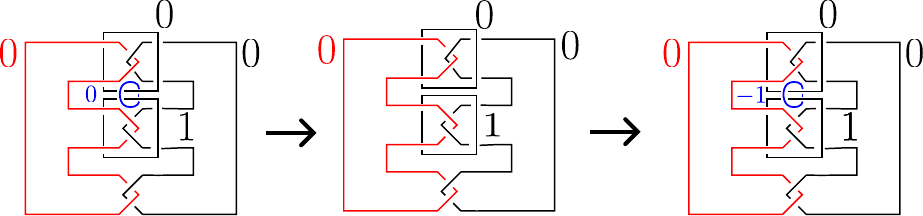}
\end{center}
\caption{\label{Fig:SeqForY1Infty1} Sequence in a Floer exact triangle $\SS^3\to Y_{0,1,\infty} \to Y_{1,\infty, 1}$. The isomorphisms with $\SS^3$ and $Y_{1,\infty,1} $ become apparent once the $1$-framed knot  is slid over the $0$-framed one thus unlinking it from the blue one. }
\end{figure}

\appendix
\section{Algebraic lemma.}\label{Appendix}
\newcommand{\FF}{\Lambda}
\newcommand{\FU}{\FF[[U]]}
Let $\FF$ be a field or $\Z$ and  $M_1, M_2, M_3$ be  finitely generated $\FU$ modules consisting of $U$-torsion elements. 
For $m\geq 0$,  $\hat R := \oplus_{i=1}^m \FU$ and $\bar R :=  \oplus_{i=1}^m \FF[[U, U^{-1}]$, denote by $p_*: \hat R \to \bar R$ the obvious $\FU$-monomorphism induced by the inclusion $\FU\to \FF[[U, U^{-1}]$.

Suppose we have the following commutative diagram with exact rows

\begin{equation}\label{LES:1}
	\begin{tikzcd}
		& 	\hat R \oplus M_1 \arrow[r,"\hat f_1"] \arrow[d,"p_*"] 
		& \hat R\oplus \hat R\oplus M_2 \arrow[r, "\hat f_2"] \arrow[d,"p_*"] 
		& \hat R\oplus M_3\arrow[d, "p_*"] \arrow[r, "\hat f_3"]
		& \hat R \oplus M_1\arrow[r, "\hat f_1"] \arrow[d,"p_*"]& \cdots\\		
		& \bar R \arrow[r, "\bar f_1"] 
		& \bar R\oplus \bar R \arrow[r, "\bar f_2"] 
		& \bar R \arrow[r, "\bar f_3"] & \bar R \arrow[r, "\bar f_1"] 
		& \cdots.\\
	\end{tikzcd}	
	\end{equation}

\begin{lemma}\label{lem:BarF_3Zero}
$\bar f_1$ is injective, $\bar f_2$ is surjective and  $\bar f_3 = 0$. 
\end{lemma}
	\begin{proof}
		Suppose first that $\FF$ is a field. Then $\bar R$ is a free finitely generated module over a PID. Consequently, after promoting $\bar f_i$ to a $\FF[U,U^{-1}]$-homomorphism,  the rank of the map  $\bar f_i$ in the exact triangle is determined by the dimension of the   $\FF[[U, U^{-1}]$-modules	in the triangle. In our case this forces $\bar f_3$ to be zero and $\bar f_1$ to be injective and $\bar f_2$ surjective.
		Now suppose that $\FF = \Z$.  Suppose for a contradiction that $Im(\bar f_3) \neq 0$. Then, since $\bar R$ is torsion free,
		$Im(\bar f_3)\otimes \R\neq 0$.  Since $\R$ is a flat $\Z$-module, by tensoring the exact triangle with $\R$, we obtain an exact triangle 
		with $\FF=\R$ a field. Thus by the previous argument  $\bar f_3\otimes id_{\R}: \bar R \otimes \R \to \bar R \otimes \R$  is the zero map
		contradicting that $Im(\bar f_3)\otimes \R\neq 0$.
	\end{proof}

\begin{lemma}\label{lemma:ConclusionFromExactSeq}	Suppose that 
	\begin{equation}\label{eq:hypOnProjectionsExSeq}
	0\to \hat R\overset{pr_{(\hat R)^2} \hat f_1}{\to} (\hat R)^2 \overset {pr_{\hat R}\hat f_2}{\to} \hat R\to 0
	\end{equation} is exact.
Then there is  an exact sequence $M_1\to M_2\to  M_3\to M_1\to\cdots$ obtained  by restricting $\hat f_1, \hat f_2, \hat  f_3$.
\end{lemma}

\begin{proof}
	Since $\hat R$ is free module over $\FU$, it is projective hence there exist $m$ linearly independent sections of $pr_{\hat R}\circ \hat f_2$, $\sigma_1,\dots, \sigma_m \in (\hat R)^2\oplus M_2$. Let $\psi_1,\dots, \psi_m \in (\hat R)^2\oplus M_2 $ be the image of the
	generators of $\hat R$ under $\hat f_1$. The hypothesis give isomorphisms
	\begin{spliteq}
		&   (\hat R)^2\oplus M_2\simeq \FU\langle \sigma_i, \psi_i\rangle_{i=1,\dots, m} \oplus M_2 =   (\hat R)^2\oplus M_2\\
		&  \hat R\oplus M_3 \simeq \FU\langle \hat f_2 (\sigma_i)\rangle_{i=1,\dots, m } \oplus M_3 = \hat R \oplus M_3. \\
	\end{spliteq}
	
		Using such isomorphisms, the upper row of \eqref{LES:1} becomes
	\begin{equation}\label{LES:2}	
	\begin{tikzcd}
		& 	\hat R \oplus M_1 \arrow[r,"A_1"] 
		& \hat R\oplus \hat R\oplus M_2 \arrow[r, "A_2"] 
		& \hat R\oplus M_3 \arrow[r, "f_3"]
		& \hat R \oplus M_1\arrow[r, "A_1"] & \cdots		
	\end{tikzcd}	
	\end{equation}
	where $A_1 = \begin{pmatrix}
								& 1& 0 \\
								& 0& 0 \\
								& 0 & C \\
							\end{pmatrix}$,
			 $A_2 = \begin{pmatrix}
								& 0 & 1 & 0  \\
								& D & 0 & F  \\
							\end{pmatrix}$ and $B,C,D,E,F$ are $\FU$-homomorphisms between the obvious modules.
	Notice that here we are using that $Hom_{\FU}(M_i, \FU) = \{0\}$	 since $M_i$ is $U$-torsion. 
	Moreover, the exactness of the sequence implies that $Im(\hat f_3) = \hat f_3 (M_3)$ and $D=0$.
		To finish the proof we need to show that  
		\begin{equation}\label{LES:3}
			M_1\overset{C}{\to} M_2\overset{F}{\to} M_3\overset{\hat f_3|_{M_3}}{\to} M_1\overset{C}{\to}\cdots,
		\end{equation}
	is an exact sequence. 
	
	\eqref{LES:3} is exact at $M_1$ because from the exactness of \eqref{LES:2}, $\ker A_1 = Im(\hat f_3) = \hat f_3(M_3) $ 
	and $\ker A_1 = \ker C \subset M_1$ because $A_1(x,y) = (x, 0, Cy)$. 
	
	\eqref{LES:3} is exact at $M_2$, indeed from $A_2\circ A_1 = 0$ we obtain that $FC = 0$ thus $Im(C)\subset \ker F$.
	To show the other inclusion, suppose that $F(x) = 0$, then $A_2(0,0,x) = 0$, hence by exactness of \eqref{LES:3}, exists $(a,b)\in \hat R \oplus M_1$, such that $(0,0,x) = A_1(a,b) = (a, 0, Cb)$, this forces $a =0$ and $x=C b$.
	
	\eqref{LES:3} is exact at $M_3$ because clearly $Im (F) \subset \ker\hat  f_3\cap M_3$.
	If $x\in M_3$ and  $\hat f_3(x) =0$ then by \eqref{LES:2}, $(0,x) = A_2(a,b,c) = (b, Fc )$ hence $b=0$ and $x=Fc$,
	hence  $x\in Im(F)$.\end{proof}

	\paragraph{Assumptions on the grading.}
	 In order to apply Lemma~\ref{lemma:ConclusionFromExactSeq}, we will make further
	assumptions on the grading of our modules.
	More precisely, we suppose that, \emph{up to degree shifts},  for some $b\in \N$,  $\hat R\simeq \HMfrom_\bullet(\#^b\SS^1\times \SS^2)$ and $\hat R\oplus \hat R\simeq \HMfrom_\bullet(\#^{b+1}\SS^1\times \SS^2)$ as  graded $\FU$-module.
	Moreover we will suppose that the map $\hat f_1$ and $\hat f_2$ have degree:
	\begin{align*}
		\gr(\hat f_1) = -1+\max\gr(\hat R \oplus \hat R) - \max \gr (\hat R)  & & &\gr(\hat f_2) = \max\gr(\hat R) - \max \gr(\hat R \oplus \hat R). 
	\end{align*}
	
	\begin{definition} We call grading assumption \textbf{AG} this set of  assumptions.
	\end{definition}\label{ass:AG}
	For the sake of simplicity, in the proofs below, we will assume that there are no degree shifts. 
	Thus $\hat R$ is identified with $\bigwedge \Z^b\otimes \FU$
	and $\hat R\oplus \hat R$ with $\bigwedge \Z^{b+1}\otimes \FU$, graded so that
	$\gr(e^{i_1}\wedge \dots \wedge e^{i_k})= -1 - k$, $\{e^j\}_{j=1}^n$ being the standard basis of $\Z^{n}$. Moreover $\gr(\hat f_1) = -1$ and $\gr (\hat f_2) = 0$.
	The modules $\bar R$ and $\bar R\oplus \bar R$ are graded similarly using $\FF[[U,U^{-1}]$ instead of $\FU$.
	We will denote the highest degree generator as $\mathbf{1}$.

	\begin{lemma} \label{lemma_algebraicCaseb0}Assume \textbf{AG}, and in addition that $b=0$.
	Then the conclusion of Lemma~\ref{lemma:ConclusionFromExactSeq} holds.
	\end{lemma}
	\begin{proof}
		We will show that the hypothesis of Lemma~\ref{lemma:ConclusionFromExactSeq} are satisfied.
		Since $b=0$, $\hat R\simeq \FU \langle\mathbf 1\rangle$  and 
			$(\hat R)^2\simeq \FU\langle \mathbf 1, e^1\rangle$.
		$\hat f_1(\mathbf1) $ is non-zero due to injectivity and for degree reasons, $pr_{\bar R^2}\hat f_1(\mathbf1) = \pm e^1$.
		Now the exactness of $\eqref{LES:1}$ implies that $\hat f_2(\mathbf 1) \neq 0$. Moreover $\hat f_2(\mathbf 1) \not \in M_3$
			for otherwise $\bar f_2(U^k \mathbf 1)  =0$ contradicting the surjectivity of $\bar f_2$. \end{proof}
			
		\begin{lemma}\label{lemma_algebraicCaseb1} Assume \textbf{AG}, and in addition that $b=1$ and either $M_1=0$ or $M_3 = 0$
		Then the conclusion of Lemma~\ref{lemma:ConclusionFromExactSeq} holds.
		\end{lemma}	
		\begin{proof} Since $b=1$, $\hat R\simeq \FU \langle\mathbf 1, e^1\rangle$  and 
			$(\hat R)^2\simeq \FU\langle \mathbf 1, e^1, e^2, e^1\wedge e^2\rangle$.
			
		Case $M_1 = 0$. In this case  $\hat f_3 = 0$. Hence $\hat f_2$ and $pr_{\hat R}\hat f_2$ are surjective. 
		The grading implies that  $pr_{\hat R}f_2(\mathbf 1) = \pm \mathbf 1$.  
		The exactness of the sequence \eqref{eq:hypOnProjectionsExSeq}, will follow if we prove that
		$\bar f_1(\mathbf 1)$ and $\bar f_1(e^1)$ are not divided by $U$.		
		For degree reasons, $\bar f_1(e^1)\in \Z \langle e^1,e^2 \rangle $  thus  it is not divided by $U$.
		On the other hand, $\bar f_1(\mathbf 1)\in \Z \langle U\mathbf 1,e^1\wedge e^2 \rangle $,
		however  $\bar f_1(\mathbf 1) $ cannot be a multiple of $U\mathbf 1$  for otherwise
			$\bar f_2(U\mathbf 1) = 0$.
		
		Case $M_3 = 0$. Since $M_3 = 0$ and $\bar f_2$ is surjective, 
		we obtain that $pr_{\hat R}\hat f_2$ is surjective too by looking at the grading of the generators.
		Now we argue as in the previous case.		
		\end{proof}
		
%

\bibliographystyle{alpha}
\bibliography{bibliography}

\vspace{0.3cm}

\end{document}